\documentclass{article}
\usepackage{amsmath,amsfonts,amssymb,amsthm}
\usepackage{mathtools}
\usepackage{fullpage}
\usepackage{enumerate}
\usepackage{tikz-cd}
\usepackage[numbers]{natbib}

\newcommand{\CR}{\mathrm{CR}}
\newcommand{\Aut}{\mathrm{Aut}}

\newcommand{\Ker}{\mathrm{Ker}}

\newcommand{\supp}{\mathrm{supp}}
\newcommand{\nul}{\mathrm{null}}
\newcommand{\id}{\mathrm{id}}
\newcommand{\A}{\mathrm{A}}

\newcommand{\Z}{\mathbb{Z}}
\newcommand{\N}{\mathbb{N}}

\newcommand{\F}{\mathbb{F}}
\newcommand{\U}{\mathcal{U}}

\renewcommand{\O}{\mathcal{O}}

\newcommand{\<}{\langle}
\renewcommand{\>}{\rangle}
\usepackage{amsbsy}

\newcommand\td[1]{\widetilde{ #1}}

\newtheorem{definition}{Definition}[section]
\newtheorem{theorem}[definition]{Theorem}

\newtheorem{lemma}[definition]{Lemma}

\newtheorem{defn}[definition]{Definition}
\newtheorem{thm}[definition]{Theorem}
\newtheorem{prop}[definition]{Proposition}
\newtheorem{lem}[definition]{Lemma}
\newtheorem{cor}[definition]{Corollary}
\newtheorem{rem}[definition]{Remark}
\newtheorem{eg}[definition]{Example}
\newtheorem{conj}[definition]{Conjectures}

\title{Projective Limits and Ultraproducts of Nonabelian Finite Groups}
\author{Nazih Nahlus, Yilong Yang}
\date{\today}

\begin{document}
\maketitle

%%%%%%%%%%%%%%%%%%%%
%%%%%%%%%%%%%%%%%%%%
%%%%%%%%%%%%%%%%%%%%

\begin{abstract}
Groups that can be approximated by finite groups have been the center of much research. This has led to the investigations of the subgroups of metric ultraproducts of finite groups. This paper attempts to study the dual problem: what are the quotients of ultraproducts of finite groups? 

Since an ultraproduct is an abstract quotient of the direct product, this also led to a more general question: what are the abstract quotients of profinite groups? Certain cases were already well-studied. For example, if we have a pro-solvable group, then it is well-known that any finite abstract quotients are still solvable.

This paper studies the case of profinite groups from a surjective inverse system of non-solvable finite groups. We shall show that, for various classes of groups $\mathcal{C}$, any finite abstract quotient of a pro-$\mathcal{C}$ group is still in $\mathcal{C}$. Here $\mathcal{C}$ can be the class of finite semisimple groups, or the class of central products of quasisimple groups, or the class of products of finite almost simple groups and finite solvable groups, or the class of finite perfect groups with bounded commutator width.

Furthermore, in the case of finite perfect groups with unbounded commutator width, we show that this is NOT the case. In fact, any finite group could be a quotient of an ultraproduct of finite perfect groups. We provide an explicit construction on how to achieve this for each finite group.
\end{abstract}

\section{Introduction}

\subsection{Background and Main Results}

Groups that can be approximated by finite groups have been the center of much research. The standard approach is to study subgroups of ``limits'' of finite groups. For example, in the terminology of Hold and Rees \cite{HR17}, sofic groups are subgroups of the metric ultraproducts of finite symmetric groups, with respect to the normalized Hamming distance (in the sense of Definition 1.6 from \cite{Thom2012}). It has gathered the attention of many works, and it is related to Kaplansky's Direct Finiteness Conjecture \cite{ES04}, Connes' embedding conjecture \cite{CLP15}, and many other problems.  As a generalization of sofic groups, one can also consider the $\mathcal{C}$-approximable groups, which are groups that can be embedded as subgroups of the metric ultraproducts of groups in a class of finite groups $\mathcal{C}$. See \cite{NST18} for further discussion.

This paper was motivated by the ``dual problem'', i.e., given an ultraproduct of finite groups, what are the possible quotients? It turns out that such ``dual'' problems have their own importance. For example, motivated by gauge field theories in physics, Zilbert asks the following question in \cite{Zilbert14}:

\begin{conj}
Can a compact simple Lie group be a quotient of an ultraproduct of finite groups?
\end{conj}

The answer to Zilbert's problem is negative, shown by \cite{NST18}. Nevertheless, the questions on quotients of ultraproducts of finite groups have not been studied much.

In a parallel manner, Bergman and Nahlus have shown the following interesting results on Lie algebra in \cite{BN11}.

\begin{thm}
Let $\mathcal{C}$ be the class of nilpotent Lie algebras, the class of solvable Lie algebras, or the class of semi-simple Lie algebras with characteristics other than 2 or 3. If we have an ultraproduct of Lie algebra s$A_n$ in a class $\mathcal{C}$, and a surjective homomorphism $f:\prod_{n\to\omega} A_n\to B$ where $B$ is a finite dimensional Lie algebra, then $B$ is also in class $\mathcal{C}$.
\end{thm}

Motivated by such results, it is natural to consider the parallel situation in the case of groups. Let $\mathcal{C}$ be a certain class of finite groups. We ask two kinds of questions:

\begin{enumerate}
\item (Weak version) If we have an ultraproduct of groups $G_n$ in a class $\mathcal{C}$, and a surjective homomorphism $f:\prod_{n\to\omega} G_n\to H$ where $H$ is a finite group, then when is $B$ also in class $\mathcal{C}$?
\item (Strong version) If we have a projective limit $G$ of a surjective inverse direct system of groups $G_n$ in a class $\mathcal{C}$, and a surjective homomorphism $f:G\to H$ where $H$ is a finite group, then when is $B$ also in class $\mathcal{C}$?
\end{enumerate}

The strong version would immediately imply the weak version since the ultraproduct of groups is a quotient of the direct product of the same groups.

The answer to the strong version of our question is already well known in certain cases. For example, if $\mathcal{C}$ is the class of finite abelian groups, then obviously any projective limit would also be abelian, and hence any quotient would also be abelian. By Corollary 4.2.3 of \cite{RZ00}, the case of $p$-groups is proven. Since each nilpotent group is a product of its $p$-Sylow subgroups, the case of nilpotent groups follows. Finally, the case of solvable groups is proven by Corollary 4.2.4 of \cite{RZ00}. For all these groups, the answer to the strong version of our question is already known to be affirmative.

In this paper, we shall prove that the strong version is true for several other cases as well.

\begin{thm}
\label{thm:mainthm1}
Let $\mathcal{C}$ be the class of finite semisimple groups, or the class of finite central products of finite quasisimple groups,  or the class of finite perfect groups with bounded commutator width. Then any finite quotient of a pro-$\mathcal{C}$ group must remain in the class $\mathcal{C}$.

If $\mathcal{C}$ is the class of finite products of finite almost simple groups, then any finite quotient of a pro-$\mathcal{C}$ group is a direct product of almost simple groups and solvable groups.
\end{thm}

In fact, we also showed that statements of this type are true for certain variants of the classes of groups above. Note that Theorem~\ref{thm:mainthm1} is dependent on the classification of finite simple groups.

However, what if we move away from the solvable groups or groups that are close to being simple? For example, if we choose $\mathcal{C}$ to be the class of finite perfect groups, then even the weak version of our question is false. For example, in Theorem 2.1.11 of \cite{HP89}, Holt and Plesken have constructed specific examples of finite perfect groups $G_n$ with huge centers, such that $G_n$ must have large commutator width. Here the commutator width of a group $G$ refers to the least integer $n$ such that every element of $G$ is a product of at most $n$ commutators. 

If a sequence of perfect groups has increasing commutator width, then it should be expected that their ``limit'' of any kind would have ``infinite commutator width'', i.e., it is probably not going to be perfect after all. Then it would have a non-trivial abelian quotient, which is not in $\mathcal{C}$. This is indeed the case, as shown in Proposition~\ref{prop:PerfectIncreaseWidth} in Section~\ref{sec:perfect} of this paper.

Another such example is constructed by Nikolov \cite{Nikolov04}, where the groups $G_n$ are even constructed to be finite perfect linear groups of dimension 15. However, it is worth noting that in both examples by Holt and Plesken (Theorem 2.1.11 of \cite{HP89}) and the one by Nikolov, $\prod_{n\to\omega}G_n$ must either have non-solvable quotients, or abelian quotients. There is nothing in between.

In pursuit of this phenomenon, Holt\footnote{https://mathoverflow.net/questions/280887/subdirect-product-of-perfect-groups} asked that, if $G$ is a subdirect product of finitely many perfect groups, must its solvable quotients be abelian? This is a ``finite'' version of our inquiry, and it was answered by Mayr and Ru\v{s}kuc in the negative \cite{MR19}. It turns out that such a quotient is not necessarily abelian, but must always be nilpotent.

In a related manner, Thiel\footnote{https://mathoverflow.net/questions/289390/inverse-limits-of-perfect-groups} asked whether all groups can be achieved as a projective limit of perfect groups. An answer was provided by de Cornulier, but his construction was through an injective directed system of infinite perfect groups.

In this regard, we would establish the following result in our paper as well.

\begin{thm}
\label{thm:mainthm2}
For any finite group $G$, $G$ is a quotient of the projective limit of a surjective inverse directed system of finite perfect groups.
\end{thm}

Note that Theorem~\ref{thm:mainthm2} is \emph{not} dependent on the classification of finite simple groups.

\subsection{Outline of the paper}
\label{subsec:layout}

In Section~\ref{sec:simple}, we shall establish the proof of Theorem~\ref{thm:mainthm1}, using mostly elementary arguments combined with the main results on ultraproducts of finite simple groups from \cite{Yang16}. We can break it down into the following subsections:

\begin{itemize}
\item In Section~\ref{subsec:prelimProf}, we establish properties of a class $\mathcal{C}$ such that any pro-$\mathcal{C}$ is a direct product of groups in $\mathcal{C}$. This shall come in handy in almost all cases below.
\item In Section~\ref{subsec:semisimple}, we establish the semisimple case of Theorem~\ref{thm:mainthm1}.
\item In Section~\ref{subsec:quasisimple}, we establish the quasisimple case of Theorem~\ref{thm:mainthm1}.
\item In Section~\ref{subsec:almostsimple}, we establish the almost simple case of Theorem~\ref{thm:mainthm1}.
\item In Section~\ref{subsec:almostsemisimple}, we define a slightly more generalized notion of almost semi-simple groups, and prove a similar result to Theorem~\ref{thm:mainthm1}.
\item In Section~\ref{subsec:bddcommwidth}, we establish the case of finite perfect groups with bounded commutator width for Theorem~\ref{thm:mainthm1}.
\end{itemize}

We shall prove Theorem~\ref{thm:mainthm2} in Section~\ref{sec:perfect}. The construction here is a bit elaborate, so we shall provide a brief breakdown of the arguments here.

We shall show that any finite group can be realized as a quotient of $\prod P_n$, a direct product of a family of finite perfect groups. First of all, if groups in the sequence $P_n$ are perfect groups with increasing commutator width, then their direct product is not perfect, and therefore they would have a non-trivial abelian quotient, as shown in the proposition below.

\begin{prop}
\label{prop:PerfectIncreaseWidth}
If $\{P_n\}_{n\in\N}$ is a sequence of finite perfect groups with strictly increasing commutator width, and $\omega$ is any non-principal ultrafilter on $\N$, then $P=\prod_{n\to\omega}P_n$ is not perfect.
\end{prop}
\begin{proof}
Since $P_n$ has strictly increasing commutator width, each $P_n$ must have a commutator width of at least $n$. Pick $g_n\in P_n$ such that it cannot be written as the product of less than $n$ commutators.

If $\lim_{n\to\omega}g_n\in P$ can be written as the product of $d$ commutators, then we must have $\{n\in\N:g_n\text{ is the product of $d$ commutators}\}\in\omega$. However, this set is finite, while $\omega$ is a non-principal ultrafilter, a contradiction. Hence $\lim_{n\to\omega}g_n$ is not in the commutator subgroup of $P$. In particular, $P$ is not perfect.
\end{proof}

We now establish the generic case. For an arbitrary finite group $G$, we would construct a surjective homomorphism from a ``pro-perfect'' group onto $G$. The homomorphism we constructed would in fact factor through an ultraproduct, as described in the statement of Theorem~\ref{thm:perfect}.

\begin{theorem}
\label{thm:perfect}
Given any finite group $G$, there is a sequence of finite perfect groups $\{P_n\}_{n\in\N}$ and an ultrafilter $\omega$ on $\N$ such that the ultraproduct $\prod_{n\to\omega} P_n$ has a surjective homomorphism onto $G$.
\end{theorem}

Since an ultraproduct is a quotient of the direct product, therefore Theorem~\ref{thm:mainthm2} is simply a corollary of Theorem~\ref{thm:perfect} above.

We first reduce the generic case to the case of a special kind of groups. In 1989, Zelmanov solved the restricted Burnside problem in \cite{Zelmanov90,Zelmanov91}. In particular, the following theorem is true.

\begin{theorem}
\label{thm:Zelmanov}
For arbitrary positive integers $d,m$, there is a unique finite group $B(d,m)$, such that any finite group of exponent $m$ on $d$ generators is isomorphic to a quotient of $B(d,m)$.
\end{theorem}

So to prove Theorem~\ref{thm:perfect}, it is enough to show that for arbitrary positive integers $d,m$, the restricted Burnside group $B(d,m)$ can be realized as an abstract quotient of an ultraproduct of finite perfect groups.

We shall provide an explicit construction of these groups $P_n$. The proof can be decomposed into the following steps: 

\begin{itemize}
\item In Section~\ref{subsec:RestrictedBurnside}, we construct and discuss the restricted Burnside coproduct, defined on finite groups with exponents dividing $m$. This is a crucial ingredient for Section~\ref{subsec:Cyclic}.
\item In Section~\ref{subsec:UltraFunctor}, we discuss some functorial properties of the ultraproduct and its interaction with the restricted Burnside coproduct. This is a crucial ingredient for Section~\ref{subsec:Cyclic}.
\item In Section~\ref{subsec:Cyclic}, we reduce the case of the restricted Burnside group $B(d,m)$ to the case of the cyclic group $\Z/m\Z$. In particular, if the cyclic group $\Z/m\Z$ could be realized as an abstract quotient of an ultraproduct of ``nice'' finite perfect groups, then so can the restricted Burnside group $B(d,m)$.
\item In Section~\ref{subsec:Gn}, we construct a sequence of finite perfect groups $G_n$.
\item In Section~\ref{subsec:Mn}, we construct a sequence of finite abelian groups $M_n$ with exponents dividing $m$, and each $M_n$ has a $G_n$ action.
\item In Section~\ref{subsec:Pn}, we show that the groups $P_n=M_n\rtimes G_n$ are perfect with increasing commutator width.
\item In Section~\ref{subsec:dual}, we establish some properties on $\prod_{n\to\omega}M_n$ and its ``dual''.
\item In Section~\ref{subsec:filter}, we construct an explicit homomorphism $\psi:\prod_{n\to\omega}M_n\to\Z/m\Z$, which could be extended to a surjective homomorphism from $\prod_{n\to\omega}P_n$ to $\Z/m\Z$.
\end{itemize}

%2.1 Preliminary in projective limits (Pro-C and Pro-PC)
%2.2 Semisimple
%2.3 Quasisimple
%2.4 Almost Simple
%2.5 Almost Semisimple
%2.6 Perfect Groups with Bounded Commutator Width
%
%
%3. Perfect Groups with Unbounded Commutator width
%
%3.1 Reduction to the case where $G$ is a restricted Burnside groups. (Need to subsectionize the first portion.)
%3.2 Categorical Properties of restricted Burnside coproduction.
%3.3 categorical Properties of the ultraproduct.
%3.4 Reduction to the case where $G$ is the cyclic group $\Z/m\Z$.
%3.5 Construction of perfect groups $G_n$.
%3.6 Constructions of abelian groups $M_n$ with exponent $m$ with a $G_n$-action.
%3.7 Perfectness and Commutator width of the groups $P_n=M_n\rtimes G_n$.
%3.8 The ultraproduct of $M_n$ and its ``dual'' group
%3.9 Construction of a homomorphism from the ultraproduct of $M_n$ to $\Z/m\Z$, and its extension to a homomorphism from the ultraproduct of $P_n$ to $\Z/m\Z$.
%
%

\subsection{Notations and Conventions}

Given a perfect group $G$ and an element $g\in G$, we define its commutator length to be the smallest integer $d$ such that it can be written as a product of at most $d$ commutators. We define the commutator width of $G$ to be the smallest integer $d$ such that every element of $G$ has commutator length at most $d$.

Given a group $G$, for elements $x,y\in G$, subgroups $H,K\leq G$, subset $S\subseteq G$, we adopt the following notations. Here $\langle S\rangle$ means the subgroup generated by the subset $S$.

\begin{align*}
g^h&\coloneqq h^{-1}gh,\\
[g,h]&\coloneqq g^{-1}h^{-1}gh,\\
[H,K]&\coloneqq \langle[g,h]\mid\text{$g\in H$ and $h\in K$}\rangle,\\
H'&\coloneqq [H,H],\\
S^{\star k}&\coloneqq \{s_1s_2\dots s_k\mid s_1,\dots,s_k\in S\}.
\end{align*}

In our paper, we adopt the convention that all group actions are from the right. Given a group $G$ acting on another group $H$ via $\phi:G\to\mathrm{Aut}(H)$, the corresponding semi-direct product is written as $H\rtimes G$. Elements of this semi-direct product group is written as $hg$ where $h\in H$ and $g\in G$, while the action of $g$ on $h$ is written as $h^g$.

For $h\in H, g\in G$, and subgroup $K\leq H$, we adopt the following conventions.

\begin{align*}
h^g&\coloneqq \phi_g(h)=g^{-1}hg,\\
[h,g]&\coloneqq h^{-1}h^g,\\
[K,G]&\coloneqq \langle[h,g]\mid h\in K,g\in G\rangle.\\
\end{align*}

Throughout this paper, all finite groups will be equipped with the discrete topology whenever needed. Let $\mathcal{C}$ be a class of finite groups equipped with the discrete topology. Then we say $G$ is a pro-$\mathcal{C}$ group if it is the inverse direct limit of a surjective system of groups in $\mathcal{C}$.

Other than the inverse direct limit, another kind of limit that concerns us is the ultraproduct. Given any index set $I$, a collection $\omega$ of subsets of $I$ is an ultrafilter if the followings are true:

\begin{enumerate}
\item If $J_1,J_2\in\omega$, then $J_1\cap J_2\in\omega$.
\item If $J_1\subseteq J_2$ as subsets of $I$, and $J_1\in\omega$, then $J_2\in\omega$.
\item For any subset $J$ of $I$, exactly one of $J$ and its complement $I-J$ is in $\omega$.
\end{enumerate}

An ultrafilter is principal if the intersection of all elements of $\omega$ is non-empty.

Given a sequence of groups $\{G_n\}_{n\in\N}$, and an ultrafilter $\omega$ on $\N$, consider the subgroup $N_\omega$ of the group $\prod_{n\in\N}G_n$ defined as $N_\omega\coloneqq\{(g_n)_{n\in\N}\in\prod_{n\in\N}G_n:\{g_n=e\}\in\omega\}$, where $e$ is the identity element. This is a normal subgroup, and we denote the quotient $\prod_{n\in\N}G_n/N_\omega$ as $\prod_{n\to\omega}G_n$. This is the ultraproduct of the sequence $\{G_n\}_{n\in\N}$ through the ultrafilter $\omega$. 

For the element $(g_n)_{n\in\N}\in\prod_{n\in\N}G_n$, its quotient image in $\prod_{n\to\omega}G_n$ is denoted as $\lim_{n\to\omega}g_n$, called the ultralimit of the sequence $(g_n)_{n\in\N}$.

For further  discussions, see \cite{BS81}.

\section{Projective limits of non-abelian groups}
\label{sec:simple}

In this section, we consider projective limits of finite groups that are close to being non-abelian finite simple. Throughout this section, we use the term \emph{semisimple group} to refer to a direct product of non-abelian finite simple groups.

\subsection{Preliminary on Profinite groups}
\label{subsec:prelimProf}

It is well-known that any projective limits of finite semisimple groups must be a direct product of finite simple groups. (E.g., see Lemma 8.2.3 of \cite{RZ00}.) In this preliminary section, we try to generalize such a result to some similar classes of groups.

Let $\mathcal{C}$ be a class of finite groups. Let $\mathcal{P}(\mathcal{C})$ be the class of finite products of groups in $\mathcal{C}$. 

Let us say that $\mathcal{C}$ is a central-simple class of groups if all groups in $\mathcal{C}$ are non-abelian, and for any homomorphism $f:G\to H$ between groups $G,H\in\mathcal{C}$, $f(G)\mathrm{C}_H(f(G))=H$ implies that $f$ is trivial or $f$ is an isomorphism. Here $\mathrm{C}_H(f(G))$ refers to the centralizer of $f(G)$. We claim that the following result is true.

\begin{prop}
\label{prop:ProPC=ProdC}
Suppose $\mathcal{C}$ is a central-simple class of finite groups. Then any pro-$\mathcal{P}(\mathcal{C})$ group is a direct product of groups in $\mathcal{C}$.
\end{prop}

Now let us prove Proposition~\ref{prop:ProPC=ProdC}. For any group $G$, let us say a normal subgroup $N$ of $G$ is $\mathcal{C}$-normal if $G/N\in\mathcal{C}$, and we write $N\trianglelefteq_\mathcal{C} G$. 

\begin{lem}
\label{lem:FiniteProj=prod}
Suppose $\mathcal{C}$ is a central-simple class of finite groups. Then for any $G\in\mathcal{P}(\mathcal{C})$, the natural homomorphism from $G$ to $\prod_{N\trianglelefteq_\mathcal{C} G}G/N$ is an isomorphism.
\end{lem}
\begin{proof}
Suppose $G\in\mathcal{P}(\mathcal{C})$, and it is the internal direct product $G=\prod_{i\in I}G_i$ where $G_i\in\mathcal{C}$ for all $i\in I$. It is enough to show that the only $\mathcal{C}$-normal subgroups are kernels of the projection maps onto some factor group.

We perform induction on the number of factor groups, i.e., on the cardinality of the index set $I$. If $I$ has only one element, then $G$ is a group in $\mathcal{C}$. If $N$ is a $\mathcal{C}$-normal subgroup of $G$, then the quotient map $q$ from $G$ to $G/N$ is a surjective homomorphism between groups in $\mathcal{C}$. In particular, we must have $q(G)\mathrm{C}_{G/N}(q(G))=G/N$, so $q$ is either trivial or an isomorphism. Since the trivial group (which is abelian) is not in $\mathcal{C}$, therefore the quotient map must be an isomorphism. Hence the only $\mathcal{C}$-normal subgroup of $G$ is the trivial subgroup, and the statement is true.

Now suppose $I$ has more than one element. Let $N$ be any $\mathcal{C}$-normal subgroup. Consider the quotient map $q:G\to G/N$. For any $i\in I$, since $G_i$ is a direct factor of $G$, any element of $q(G_i)$ must commute with any element of $q(G_j)$ for $j\neq i$. In particular, we have $q(G_i)\mathrm{C}_{G/N}(q(G_i))=G/N$. So $q$ restricted to $G_i$ must either be trivial or an isomorphism.

If $q$ is trivial on all $G_i$, then $G/N$ is trivial and not in $\mathcal{C}$, which is impossible. Hence $q$ restricts to an isomorphism to some $G_i$. Furthermore, if $q$ restricted to $G_i$ and $q$ restricted to $G_j$ are both isomorphisms for some $i\neq j$, then all elements in $q(G_i)=G/N$ and in $q(G_j)=G/N$ must commute, so $G/N$ is an abelian group and not in $\mathcal{C}$, also impossible. Hence, there is a unique $i\in I$ such that $q$ restricted to $G_i$ is an isomorphism, and $q$ kills all other $G_j$ where $j\neq i$.

In particular, $q$ is a projection map onto a factor group $G_i$. So we are done.
\end{proof}

\begin{proof}[Proof of Proposition~\ref{prop:ProPC=ProdC}]
Let $G$ be a pro-$\mathcal{P}(\mathcal{C})$ group. Then it is the projective limits of $G/U\in \mathcal{P}(\mathcal{C})$ for all open $\mathcal{P}(\mathcal{C})$-normal subgroups $U$.

But for each open $\mathcal{P}(\mathcal{C})$-normal subgroup $U$, by Lemma~\ref{lem:FiniteProj=prod}, we have an isomorphism $G/U\simeq \prod_{N\in\mathcal{M}_U} G/N$ where $\mathcal{M}_U$ is the set of open $\mathcal{C}$-normal subgroups containing $U$. By taking inverse limits of these isomorphisms, we have an isomorphism $G\simeq\prod_{N\in\mathcal{M}} G/N$ where $\mathcal{M}_U$ is the set of open $\mathcal{C}$-normal subgroups of $G$.
\end{proof}

\subsection{Semisimple groups}
\label{subsec:semisimple}

We start with projective limits of finite semisimple groups. Note that, as we have already noted, this must be a direct product of non-abelian finite simple groups. Such groups are also sometimes called semisimple profinite groups. Let us first describe all maximal normal subgroups of a semisimple profinite group.

Given a semisimple profinite group $G=\prod_{i\in I}S_i$ for a family of non-abelian finite simple groups $S_i$, for any subset $J\subseteq I$, we define $G_J:=\{(g_i)_{i\in I}\in G\mid g_i=e\text{ if $i\notin J$}\}$, i.e., the subgroup of elements whose ``support'' is inside $J$. Clearly $G$ is the internal direct product of $G_J$ and $G_{I-J}$.

For any ultrafilter $\U$ on $I$, we define $N_\U:=\{(g_i)_{i\in I}\in G\mid\{i\in I\mid g_i=e\}\in\U\}$. Note that the ultraproduct of the groups $S_i$ is by definition $G/N_\U$.

\begin{prop}
\label{prop:SemisimpleProfiniteMaximal}
If $G=\prod_{i\in I}S_i$ for a family of non-abelian finite simple groups $S_i$, and $M$ is any abstract maximal normal subgroup of $G$, then $N_\U\subseteq M$ for some ultrafilter $\U$ on $M$.
\end{prop}
\begin{proof}
Note that if $G$ is an arbitrary product of non-abelian finite simple groups, then each factor group of $G$ has commutator width $1$ by the classification of non-abelian finite simple groups. So $G$ has commutator width $1$, and thus it is perfect. So $G/M$ is also a perfect group.

Since $M$ is maximal, $G/M$ is perfect and simple. So by Corollary~\ref{cor:UltraFactorOntoPerfectSimple} below, if $G=\prod_{i\in I}S_i$ for a family of non-abelian finite simple groups $S_i$, then the quotient map $G\to G/M$ must factor through some ultraproduct $\prod_{i\to\U}S_i$ for some ultrafilter $\U$ on $I$. So we are done.
\end{proof}

\begin{cor}
\label{cor:SemisimpleProfiniteMaximal}
If $G=\prod_{i\in I}S_i$ for a family of non-abelian finite simple groups $S_i$, and $M$ is any abstract maximal normal subgroup of $G$ with finite index, then $M=N_\U$ for some ultrafilter $\U$ on $I$, and $G/M$ is isomorphic to $S_i$ for some $i$.
\end{cor}
\begin{proof}
By Proposition~\ref{prop:SemisimpleProfiniteMaximal}, we know that we can find an ultrafilter $\U$ such that $N_\U\subseteq M$. But by \cite{Yang16}, $G/N_\U$ is either isomorphic to $S_i$ for some $i$, or it cannot have any non-trivial finite dimensional unitary representation. But $G/M$ is a finite quotient of $G/N_\U$, which must have a non-trivial finite dimensional unitary representation. Hence $G/N_\U$ is either isomorphic to $S_i$ for some $i$, and thus $M=N_\U$ and $G/M$ is isomorphic to $S_i$ for some $i$.
\end{proof}

\begin{cor}
\label{cor:IntersectionMaximalPerfect}
If $G=\prod_{i\in I}S_i$ for a family of non-abelian finite simple groups $S_i$, then intersections of maximal normal subgroups of finite indices of $G$ is a perfect group.
\end{cor}
\begin{proof}
Using Corollary~\ref{cor:SemisimpleProfiniteMaximal}, let these maximal normal subgroups of finite indices be $\{N_{\U_k}\}_{k\in K}$ for some index set $K$ and ultrafilters $\U_k$ on $I$. For each $g=(g_i)_{i\in I}\in G$, let $\mathrm{supp}(g)=\{i\in I:g_i\neq e\}$. Then $g\in N_{\U_k}$ if and only if $\mathrm{supp}(g)\notin\U_k$. So $g\in\bigcap_{k\in K}N_{\U_k}$ if and only if $\mathrm{supp}(g)\notin\bigcup_{k\in K}\U_k$.

In particular, $\bigcap_{k\in K}N_{\U_k}$ is generated by $G_J$ where $J\notin\bigcup_{k\in K}\U_k$. Since each $G_J$ is also a semisimple profinite group, hence a perfect group, therefore the subgroup they generate must also be perfect. So we are done.
\end{proof}

Now we can establish our desired result.

\begin{prop}
\label{prop:SimpleDirect}
If $G$ is a semisimple profinite group, then any finite quotient of $G$ is a finite semisimple group.
\end{prop}
\begin{proof}
Say $G=\prod_{i\in I}S_i$ for a family of non-abelian finite simple groups $S_i$. Suppose for contradiction that our statement is not true, and let $B$ be a minimal counterexample, i.e., $B$ is a finite quotient of $G$, but it is \emph{not} semisimple. Let $q:G\to B$ be the quotient map.

If $B$ has a non-trivial center $Z$, then by minimality, $B/Z$ is a finite semisimple group. Therefore, $Z$ is an intersection of maximal normal subgroups of $B$ with finite indices. Pulling back to $G$, we see that $q^{-1}(Z)$ is an intersection of maximal normal subgroups of $B$ with finite indices. Therefore according to Corollary~\ref{cor:IntersectionMaximalPerfect}, it is a perfect group. Hence its image $Z$ must also be a perfect group. But $Z$ is also abelian, hence it must be trivial, a contradiction.

So $B$ must be centerless. Furthermore, if $B$ is the direct product of two non-trivial groups, then by minimality of $B$, the two factor groups must both be semisimple groups. Hence $B$ itself is a semisimple group, a contradiction.

So $B$ must be a centerless group, and it is not the direct product of two non-trivial groups. By Corollary~\ref{cor:UltraEquiv} below, this means that $q:G\to B$ must factor through an ultraproduct. But by \cite{Yang16} again, this means that $B$ is either infinite, or $B$ is isomorphic to $S_i$ for some $i$. Neither case is possible.

To sum up, such a counterexample is impossible. 
\end{proof}

Here are the necessary lemmas for the proofs above. The idea is adapted from Bergman and Nahlus \cite{BN11}, who proved a similar statement for Lie algebras.

\begin{lemma}[{\cite[Lemma 7]{BN11}}]
\label{lem:Ultra}
Suppose $(A_i)_{i\in I}$ is a family of sets, $B$ is a set, and $f:\prod_{i\in I} A_i\to B$ is a set map whose image has more than one element. Then the following are equivalent:
\begin{enumerate}[(a)]
\item For every subset $J\subseteq I$, the map $f$ factors either through the projection $\prod_{i\in I}A_i\to\prod_{i\in J}A_i$, or through the projection $\prod_{i\in I}A_i\to\prod_{i\in I-J}A_i$.
\item The map $f$ factors through the natural map $\prod_{i\in I}A_i\to\prod_{i\in I}A_i/\U$, where $\U$ is an ultrafilter on the index set $I$.
\end{enumerate}
When these hold, the ultrafilter $\U$ is uniquely determined by the map $f$.
\end{lemma}

\begin{cor}[Analogue of {\cite[Lemma 5]{BN11}}]
\label{cor:UltraEquiv}
Let $B$ be a non-trivial group. Then the following are equivalent:
\begin{enumerate}[(i)]
\item Every surjective homomorphism $f:A_1\times A_2\to B$ from a direct product of two groups $A_1,A_2$ onto $B$ factors through the projection of $A_1\times A_2$ onto $A_1$ or the projection of $A_1\times A_2$ onto $A_2$.
\item There are no non-trivial normal subgroups $N_1,N_2$ of $B$ such that $B=N_1N_2$ and $[N_1,N_2]=\{e\}$.
\item The center $Z(B)$ of $B$ is trivial, and $B$ is not the direct product of two non-trivial groups.
\item For every surjective homomorphism $f:\prod_{i\in I}G_i\to B$ from a direct product of a family of groups $G_i$ over any index set $I$ onto $B$, $f$ factors through a quotient map from $\prod_{i\in I}G_i$ onto some ultraproduct $\prod_{i\in I}G_i$ over some ultrafilter $\U$ on $I$.
\end{enumerate}
\end{cor}
\begin{proof}
First let us show that the negation of $(i)$, $(ii)$, and $(iii)$ are equivalent.

Suppose $(i)$ is false. Then let $f:A_1\times A_2\to B$ be a counterexample. Let $N_1=f(A_1\times\{e\})$ and $N_2=f(\{e\}\times A_2)$. Then $N_1$ and $N_2$ are non-trivial. Since $f$ is surjective, $N_1$ and $N_2$ are normal subgroups of $B$, and $B=N_1N_2$. Finally, 
\[
[N_1,N_2]=[f(A_1\times\{e\}),f(\{e\}\times A_2)]=f([A_1\times\{e\},\{e\}\times A_2])=\{e\}.\] 
So $(ii)$ is false.

Suppose $(ii)$ is false. Let $N_1,N_2$ be a pair of normal subgroups of $B$ and a counterexample to $(ii)$. If $N_1\cap N_2=\{e\}$, then $B=N_1\times N_2$, and $(iii)$ is false. Suppose $N_1\cap N_2$ is non-trivial. For any element $a\in N_1\cap N_2$, since $a\in N_2$ and $[N_1,N_2]=\{e\}$, $a$ must commute with every element of $N_1$. Similarly, since $a\in N_1$, it must commute with every element of $N_2$. And since $B=N_1N_2$, we must have $a\in Z(B)$. So $Z(B)\supseteq N_1\cap N_2$ is non-trivial. So $(iii)$ is false.

Now suppose $(iii)$ is false. If $B$ is the direct product of two non-trivial groups, then obviously $(i)$ is false. Suppose $Z(B)$ is non-trivial. Then let $f:Z(B)\times B\to B$ be the homomorphism induced by the inclusion homomorphism of $Z(B)$ into $B$ and the identity homomorphism of $B$. This provides a counterexample of $(i)$.

Now we shall show that $(i)$ and $(iv)$ are equivalent.

$(i)$ is a special case of $(iv)$ when $I=\{1,2\}$. So we only need to show that $(i)$ implies $(iv)$. But this follows easily from Lemma~\ref{lem:Ultra}.
\end{proof}

\begin{cor}
\label{cor:UltraFactorOntoPerfectSimple}
If $B$ is a perfect simple group, then for any surjective homomorphism $f:\prod_{i\in I}G_i\to B$ from a direct product of a family of groups $G_i$ over any index set $I$ onto $B$, $f$ factors through a quotient map from $\prod_{i\in I}G_i$ onto some ultraproduct $\prod_{i\to\U}G_i$ over some ultrafilter $\U$ on $I$.
\end{cor}
\begin{proof}
Since $B$ is perfect and simple, it cannot have a non-trivial center, and it cannot be the direct product of two non-trivial subgroups. So the statement is true by Corollary~\ref{cor:UltraEquiv}.
\end{proof}

\subsection{Quasisimple groups}
\label{subsec:quasisimple}

A group is quasisimple if it is a perfect central extension of a non-abelian finite simple group. Analogously, one can define a ``quasi-semisimple group'' to be a perfect central extension of a finite semisimple group. It turns out that such a group is simply a central product of quasisimple groups.

\begin{lem}
\label{lem:PerfectCentralExt=CentralProdQuasi}
A finite perfect central extension of a finite semisimple group is a central product of finite quasisimple groups.
\end{lem}
\begin{proof}
Suppose $G$ is a perfect central extension of a finite semisimple group $S$, which is the internal direct product of non-abelian finite simple subgroups $S_1,\dots,S_k$. Let the quotient map be $q:G\to S$, let $M_i=q^{-1}(S_i)$, i.e., the pull-back in the following diagram.

\begin{figure}[h!]
\centering
\begin{tikzcd}
M_i\arrow[r]\arrow[d] & Gi\arrow[d]\\
S_i\arrow[r] & S
\end{tikzcd}
\end{figure}

Let $Q_i$ be the unique maximal perfect normal subgroup of $M_i$. Now each $Q_i$ is perfect by construction, and since $M_i$ is finite, $M_i/Q_i$ is solvable. Therefore $q(Q_i)=S_i$. So, we have the following diagram.

\begin{figure}[h!]
\centering
\begin{tikzcd}
Q_i\arrow[r]\arrow[d] & M_i\arrow[r]\arrow[d] & Gi\arrow[d]\\
S_i\arrow[r] & S_i\arrow[r] & S
\end{tikzcd}
\end{figure}

%Suppose $G$ is a perfect central extension of a finite semisimple group $S$, which is the internal direct product of non-abelian finite simple subgroups $S_1,\dots,S_k$. Let the quotient map be $q:G\to S$, let $M_i=q^{-1}(S_i)$, and let $Q_i$ be the unique maximal perfect normal subgroup of $M_i$.

Finally, since the kernel of $q$ is the center of $G$, we see that $Q_i$ must be a perfect central extension of $S_i$, i.e., it is quasisimple. I now claim that $G$ is the central product of all these $Q_i$.

Let $Z$ be the center of $G$. Since $q(Q_i)=S_i$, we see that $q$ restricted to $Q_1\dots Q_k$ is already surjective onto $S\simeq G/Z$. Hence we must have $G=Q_1\dots Q_kZ$. 

If $Q_1\dots Q_k$ is a proper normal subgroup of $G$, then $G/(Q_1\dots Q_k)$ is non-trivial and abelian, which cannot happen since $G$ is perfect. So $G=Q_1\dots Q_k$. It remains to be shown that $[Q_i,Q_j]$ is trivial for all $i\neq j$.

Now $Q_j$ acts on $Q_i$ via conjugation, but $[Q_i,Q_j]\subseteq Q_i\cap Q_j$ is in the kernel of $q$, hence inside of the center $Z$. So $Q_j$ acts trivially on $Q_i/(Q_i\cap Z)\simeq S_i$. But any automorphism of a finite quasisimple group must be trivial if it acts trivially on the corresponding simple group. So $[Q_i,Q_j]$ is in fact trivial.

So $G$ is indeed the central product of $Q_1,\dots,Q_k$.
\end{proof}

In general, a direct product of quasisimple groups might have a quotient that is not quasisimple anymore, but rather a central product of quasisimple groups, as shown in the following proposition.

\begin{prop}
\label{prop:QuasisimpleDirect}
If $G$ is an arbitrary direct product of finite quasisimple groups, then any finite quotient must be a central product of finite quasisimple groups.
\end{prop}
\begin{proof}
Say $G=\prod_{i\in I}Q_i$ for a family of finite quasisimple groups $Q_i$. Let $Z_i$ be the center of $Q_i$, and set $Z=\prod_{i\in I}Z_i$. Then $Z$ is the center of $G$, and $G/Z$ is a semisimple profinite group.

Suppose we have a surjective homomorphism $q:G\to B$. Then $B/q(Z)$ is a finite quotient of the semisimple profinite group $G/Z$. Hence $B/q(Z)$ is semisimple. Furthermore, since $Z$ is the center of $G$ and $q$ is surjective, $q(Z)$ is inside the center of $B$. So $B$ is a central extension of the semisimple group $B/q(Z)$.

Finally, by classification of finite simple groups, any finite quasisimple group has commutator width at most $2$. So $G$ has commutator width at most $2$. So $G$ is perfect, and its quotient $B$ is also perfect.

So $B$ is a finite perfect central extension of a finite semisimple group. 
\end{proof}

Now, the class of quasisimple groups is not central-simple (proper quotients of a quasisimple group might still be quasisimple). So the result above does not immediately generalize to projective limits that are not direct products.

However, by the theory of Schur multipliers, each non-abelian finite simple group has a unique maximal perfect central extension, i.e., the universal covering group for the non-abelian finite simple group. (See, e.g., Proposition 33.4 of \cite{Aschbacher00}.) And each quasisimple group must be a quotient of such a group. 

These universal covering groups (i.e., quasisimple groups with trivial Schur multiplier) are also sometimes called superperfect quasisimple groups. Then we have the following result.

\begin{prop}
The class of finite superperfect quasisimple groups is central-simple.
\end{prop}
\begin{proof}
Suppose $G,H$ are finite superperfect quasisimple groups, and $f:G\to H$ is any group homomorphism such that $f(G)\mathrm{C}_H(f(G))=H$. Since the subgroups $f(G)$ and $\mathrm{C}_H(f(G))$ both normalize $f(G)$, therefore $f(G)\mathrm{C}_H(f(G))=H$ implies that $f(G)$ is normal in $H$.

If $f(G)\neq H$, then note that all proper normal subgroups of a quasisimple group are central, so $f(G)$ is abelian. But since $G$ is perfect, $f(G)$ must be trivial.

If $f(G)=H$, then $G$ is a central perfect extension of $H$. But since $H$ is super perfect, it is already a maximal central perfect extension, hence $f$ can only be an isomorphism.
\end{proof}

\begin{prop}
\label{prop:QuasisimpleDirect}
If $G$ is a projective limit of finite central products of finite quasisimple groups, then any finite quotient must be a central product of finite quasisimple groups.
\end{prop}
\begin{proof}
Let $\mathcal{C}$ be the class of superperfect quasisimple groups. Then for each finite central product $Q$ of finite quasisimple groups $\{Q_i\}_{i\in I}$, then $Q$ is a quotient of the direct products of these $Q_i$ by a central subgroup. Let $\widetilde{Q}_i$ be the universal perfect central extension of $Q_i$, then $Q$ is a central quotient of $\widetilde{Q}=\prod_{i\in I}\widetilde{Q}_i\in\mathcal{P}(\mathcal{C})$.

Now let $G$ be the projective limit of a directed family of finite central products of finite quasisimple groups $\{G_i\}_{i\in I}$. Let $\widetilde{G}_i$ be the universal perfect central extension of $G_i$, so that $\widetilde{G}_i\in\mathcal{P}(\mathcal{C})$. Let $Z_i$ be the kernel of the quotient map from $\widetilde{G}_i$ to $G_i$. 

Note that any surjective homomorphism $\phi_{ij}:G_i\to G_j$ must induce a surjective homomorphism $\widetilde{\phi}_{ij}:\widetilde{G}_i\to \widetilde{G}_j$, and $\widetilde{\phi}_{ij}$ must also send the center into the center. So the directed family $\{G_i\}_{i\in I}$ must induce directed families $\{\widetilde{G}_i\}_{i\in I}$ and $\{Z_i\}_{i\in I}$, and thus we have corresponding induced projective limits $\widetilde{G}$ and $\widetilde{Z}$.

Note that the projective limit on inverse systems of finite groups is an exact functor.  (See, e.g., Proposition 2.2.4 of \cite{RZ00}.) Therefore, $G$ is a quotient of $\widetilde{G}$ by the closed subgroup $Z$. Since $\widetilde{G}$ is a pro-$\mathcal{P}(\mathcal{C})$ group and $\mathcal{C}$ is a central-simple class, $\widetilde{G}$ must be a direct product of superperfect quasisimple groups, and thus any finite quotient of $\widetilde{G}$ must be a central product of finite quasisimple groups.

Since $G$ is a quotient of $\widetilde{G}$, the result holds for $G$ as well.
\end{proof}

Now, note that the class of finite central products of finite quasisimple groups is NOT closed under sub-direct product, so it is not a formation of groups.

\begin{eg}
Consider the group $\mathrm{SL}_2(5)$ of $2$ by $2$ matrices over the field of order $5$, whose determinants are $1$. It is quasisimple over the finite simple group $\A_5$. Then $\mathrm{SL}_2(5)\times\mathrm{SL}_2(5)$ is a perfect central extension of $\A_5\times\A_5$. Now take the diagonal subgroup of $\A_5\times\A_5$, and take its pre-image in $\mathrm{SL}_2(5)\times\mathrm{SL}_2(5)$. This is a sub-direct product of finite quasisimple groups, but it is isomorphic to $\mathrm{SL}_2(5)\times(\Z/2\Z)$, NOT a central product of finite quasisimple groups.
\end{eg}

So if one prefers to use a formation of groups, we can consider a larger class of groups, closed under sub-direct product.

\begin{prop}
The following three classes of groups are the same:
\begin{enumerate}
\item $\mathcal{C}$ is the smallest class of groups containing all finite quasisimple groups which is closed under sub-direct products and quotients.
\item $\mathcal{C}$ is the class of finite central extensions of finite semisimple groups.
\item $\mathcal{C}$ is the class of finite central products of finite quasisimple groups and finite abelian groups.
\end{enumerate}
\end{prop}
\begin{proof}
Let us call the three classes $\mathcal{C}_1,\mathcal{C}_2,\mathcal{C}_3$.

Since central products are quotients of the direct products, $\mathcal{C}_1$ is closed under central products. Also note that any cyclic group $Z$ can be realized as a center for some quasisimple group over $\mathrm{PSL}_n(q)$ for some $n,q$. Let $Q$ be such a quasisimple group over a non-abelian finite simple group $S$ with center $Z$. Then $Q^2$ has a quotient $S^2$, and the pullback of the diagonal subgroup of $S^2$ in $Q^2$ is a sub-direct product of two copies of $Q$, and it is isomorphic to $Q\times Z$. Hence $\mathcal{C}_1$ must also contain all finite abelian groups. So $\mathcal{C}_3\subseteq\mathcal{C}_1$.

%It is also easy to see that a finite central product of finite quasisimple groups and finite abelian groups must be a finite central extension of finite semisimple groups. So $\mathcal{C}_3\subseteq\mathcal{C}_2$.

Suppose $G\in\mathcal{C}_2$. Then we have a quotient map $q:G\to S$ such that $S$ is semisimple, and the kernel of $q$ is the center $Z$ of $G$. Since $S$ is perfect, the unique largest perfect normal subgroup $N$ in $G$ must have $q(N)=S$, and the restriction of $q$ to $N$ must have kernel $N\cap Z$, which is in the center of $N$. So $N$ is a perfect central extension of $S$. In particular, by Lemma~\ref{lem:PerfectCentralExt=CentralProdQuasi}, $N$ is a finite central product of finite quasisimple groups.

But since $q(N)=S$ and $Z$ is the kernel of $q$, we have $G=NZ$. And since $Z$ is the center, $[N,Z]$ must be trivial. So $G$ is the central product of $N$ and the abelian group $Z$. So $\mathcal{C}_2\subseteq\mathcal{C}_3$.

It remains to show that $\mathcal{C}_1\subseteq\mathcal{C}_2$. Since $\mathcal{C}_2$ contains all finite quasisimple groups, we only need to show that it is closed under sub-direct products and quotients. 

The class $\mathcal{C}_2$ is obviously closed under quotients. We now prove that sub-direct products of groups in $\mathcal{C}_2$ are still in $\mathcal{C}_2$.

Suppose $G$ has two normal subgroups $N_1,N_2$ such that $N_1\cap N_2$ is trivial, and $G_i=G/N_i\in\mathcal{C}_2$ for $i=1,2$. Then $G$ is a finite group. Let $M$ be the intersection of all maximal normal subgroups of $G$ whose corresponding quotients are non-abelian finite simple groups. Then $G/M$ is a finite semisimple group. We need to show that $M$ is in the center of $G$.

Consider the quotient map $q_i:G\to G_i$. Then $q_i(M)$ is in the intersection of all maximal normal subgroups of $G_i$ whose corresponding quotients are non-abelian finite simple groups. Since $G_i\in\mathcal{C}_2$, $q_i(M)$ is the center of $G_i$. In particular, for any $g\in G$ and $x\in M$, we have $xgN_i=gxN_i$. So, since $N_1\cap N_2$ is trivial, we have 
\[
\{xg\}=(xgN_1)\cap(xgN_2)=(gxN_1)\cap(gxN_2)=\{gx\}.\] 

So $x$ is in the center of $G$. So $G$ is a central extension of $G/M$, and $G\in\mathcal{C}_2$. So we conclude that $\mathcal{C}_1\subseteq\mathcal{C}_2$.
\end{proof}

Now we have a corresponding result on projective limits of finite central extensions of finite semisimple groups.

\begin{prop}
If $G$ is a projective limit of a family of finite central extensions of finite semisimple groups, then any finite quotient of $G$ is still a finite central extension of a finite semisimple group.
\end{prop}
\begin{proof}
Now if $G$ is a projective limit of a family of groups $G_i$, each a finite central extension of a finite semisimple group $S_i$. Let $Z_i$ be the center of $G_i$. Then we have a corresponding projective limit $Z$ of the abelian groups $Z_i$, and a projective limit $S$ of the semisimple groups $S_i$. Note that the projective limit on inverse systems of finite groups is an exact functor.  (See, e.g., Proposition 2.2.4 of \cite{RZ00}.) So $Z$ is a closed central subgroup of $G$ and $G/Z\simeq S$, a semisimple profinite group.

Suppose we have a surjective homomorphism $q:G\to B$ where $B$ is finite. Then $q(Z)$ is a central subgroup of $B$, and $B/q(Z)$ is a finite quotient of the semisimple profinite group $G/Z$. By Proposition~\ref{prop:SimpleDirect}, $B/q(Z)$ is a finite semisimple group. So $B$ is a finite central extension of a finite semisimple group.
\end{proof}

\subsection{Projective limits via almost simple groups}
\label{subsec:almostsimple}

We now briefly introduce the concept of almost simple groups.

\begin{defn}[Definition of CR-radicals]
Given any finite group $G$, we define its CR-radical $\CR(G)$ to be the largest semisimple normal subgroup. (``CR'' stands for ``completely reducible''.)
\end{defn}

The CR-radical uniquely exists for all groups, and a solvable group must have a trivial CR-radical.

\begin{defn}[Definition of Almost Simple Groups]
Given any finite group $G$, we say it is almost simple if $\CR(G)$ is a non-abelian finite simple group, and $G$ has no non-trivial solvable normal subgroups.
\end{defn}

Note that the above definition is equivalent to the usual definition of an almost simple group, which is a subgroup of $\mathrm{Aut}(S)$ containing $\mathrm{Inn}(S)$ for some non-abelian finite simple group $S$.

\begin{prop}[Alternative Definition of Almost Simple Groups]
If $G$ is almost simple, then it is isomorphic to a subgroup of $\mathrm{Aut}(S)$ containing $\mathrm{Inn}(S)$ for some non-abelian finite simple group $S$.
\end{prop}
\begin{proof}
Let $S$ be the CR-radical of $G$. Then $G$ acts on $S$ by conjugation, which induces a homomorphism from $G$ to $\mathrm{Aut}(S)$. Let $N$ be the kernel of this homomorphism. Assume for contradiction that it is not trivial.

Since $N$ is a normal subgroup of $G$ and $\CR(N)$ is characteristic in $N$, therefore $\CR(N)$ is a semisimple normal subgroup of $G$. Therefore $\CR(N)\subseteq\CR(G)=S$ and by construction, every element of $\CR(N)$ commutes with every element of $S$. Hence $\CR(N)$ is in the center of $S$. But $S$ is a non-abelian finite simple group. So $\CR(N)$ is trivial.

Therefore, any minimal normal subgroup of $N$ must be a direct product of cyclic subgroups of prime order. In particular, the solvable radical $M$ of $N$ (the unique largest solvable normal subgroup of $N$) is non-trivial. Since $M$ is characteristic in $N$, it is a non-trivial solvable normal subgroup of $G$, a contradiction.

So we must conclude that $N$ is trivial. Hence $G$ is isomorphic to a subgroup of $\mathrm{Aut}(S)$. And since $G$ contains $S$, the image of $G$ in $\mathrm{Aut}(S)$ would contain $\mathrm{Inn}(S)$.
\end{proof}

%It can be shown that, if $G$ is almost simple, then it is isomorphic to a subgroup of $\mathrm{Aut}(S)$ containing $\mathrm{Inn}(S)$ for some non-abelian finite simple group $S$. Its CR-radical in this case would be isomorphic to $S$. In particular, its CR-subgroup is the unique minimal normal subgroup. 

By the classification of finite simple groups, we also see that if $G$ is almost simple, then $G/\CR(G)$ is solvable.

\begin{prop}
\label{prop:AlmostSimpleCentralSimp}
The class of finite almost simple groups is central-simple.
\end{prop}
\begin{proof}
Suppose $G,H$ are finite almost simple groups, and $f:G\to H$ is any group homomorphism such that $f(G)\mathrm{C}_H(f(G))=H$. Then $f(G)$ is a normal subgroup of $H$. 

Suppose it is not trivial. Let $S$ be the unique minimal normal subgroup of $H$. Then it is a finite simple group, and any non-trivial normal subgroup of $H$ must contain $S=\CR(H)$. So $f(G)$ contains $S$. However, since $H$ can be seen as a subgroup of $\mathrm{Aut}(S)$, the centralizer of $S$ in $H$ is trivial. Hence $\mathrm{C}_H(f(G))$ is trivial. So  $f(G)\mathrm{C}_H(f(G))=H$ implies that $f(G)=H$.

Finally, note that proper quotients of an almost simple group must be solvable. But $f(G)=H$ is not solvable. So $f$ has a trivial kernel, and thus it is an isomorphism.
\end{proof}

As a result of Proposition~\ref{prop:AlmostSimpleCentralSimp}, projective limits of finite products of almost simple groups must be a direct product of almost simple groups.  Now, the quotient of such a profinite group might no longer be almost simple. For example, if we take a single almost simple group that is not simple itself, then it will have a solvable quotient. So we need to look at a larger class of groups.

Let $\mathcal{C}$ be the class of finite groups that are direct products of solvable groups and almost simple groups. Note that this class is closed under central products.

\begin{prop}
\label{prop:AlmostSimpleCloseCentral}
If $G,H\in\mathcal{C}$, then any central product of $G,H$ is still in $\mathcal{C}$.
\end{prop}
\begin{proof}
Note that $\mathcal{C}$ is obviously closed under direct products. So we only need to show that it is closed under central quotients.

Suppose $G\in\mathcal{C}$, say $G$ is the internal direct product of almost simple subgroups $G_1,\dots,G_k$ and a solvable subgroup $S$. Now $G$ acts on its normal subgroup $\CR(G_1)\dots\CR(G_k)$ by conjugation. Since $G_1,\dots,G_k$ are all almost simple, the kernel of the induced homomorphism $G\to\mathrm{Aut}(\CR(G_1)\dots\CR(G_k))$ is exactly $S$. So the center $Z$ of $G$ is contained inside of $S$. Hence $G/Z\simeq G_1\times\dots\times G_k\times (S/Z)\in\mathcal{C}$.
\end{proof}

In the remainder of this section we establish the following:

\begin{prop}
\label{prop:AlmostSimpleProjectiveLimit}
Let $G$ be a pro-$\mathcal{C}$ group. Then any finite homomorphic image of $G$ must still be in $\mathcal{C}$.
\end{prop}
\begin{proof}
Consider a potential counterexample. Suppose for contradiction that we have a pro-$\mathcal{C}$ group $G=\prod_{i\in I} G_i$, each $G_i$ is either finite solvable or finite almost simple, and $\phi:G\to B$ is a surjective homomorphism onto a finite group $B$ NOT in $\mathcal{C}$. Since $B$ is finite, if such counterexample exists, then we can in fact let $B$ be a minimal counterexample.

Then by Corollary~\ref{cor:BdoesNotExist} below, $B$ cannot exist. Hence we are done.
\end{proof}

In the remainder of this section, we preserve the notations of $G=\prod_{i\in I} G_i$ and $\phi:G\to B$ as in the minimal counterexample in the proof above.

For any subset $J$ of $I$, let $G_J=\{(g_i)_{i\in I}: \text{$g_i=e$ if $i\notin J$}\}$. 

We use the notation $\CR(G_J)=\{(g_i)_{i\in I}: \text{$g_i=e$ if $i\notin J$ and $g_i\in\CR(G_i)$ if $i\in J$}\}$. Note that $G_J/\CR(G_J)$ is isomorphic to $\prod_{i\in J}(G_i/\CR(G_i))$, which is a pro-solvable group. We use $\CR(G)$ to denote $\CR(G_I)$.

%To prove this, we consider potential counterexamples. For the rest of this section, we fix a pro-$\mathcal{C}$ group $G=\prod_{i\in I} G_i$, each $G_i$ are either finite solvable or finite almost simple, and $\phi:G\to B$ is a surjective homomorphism onto a finite group $B$ NOT in $\mathcal{C}$. Since $B$ is finite, if such counterexample exists, then we can in fact let $B$ be a minimal counterexample.
%
%For any subset $J$ of $I$, let $G_J=\{(g_i)_{i\in I}: \text{$g_i=e$ if $i\notin J$}\}$. 
%
%We use the notation $\CR(G_J)=\{(g_i)_{i\in I}: \text{$g_i=e$ if $i\notin J$ and $g_i\in\CR(G_i)$ if $i\in J$}\}$. Note that $G_J/\CR(G_J)$ is isomorphic to $\prod_{i\in J}(G_i/\CR(G_i))$, which is a pro-solvable group. We use $\CR(G)$ to denote $\CR(G_I)$.

\begin{lem}[Dichotomy between complementary factors]
For any subset $J$ of $I$, either $\phi(G_J)=B$ or $\phi(G_{I-J})=B$. In particular, one is $B$ and the other one is inside the center of $B$.
\end{lem}
\begin{proof}
Note that $G$ is the internal direct product of $G_J$ and $G_{I-J}$, hence $B=\phi(G)$ is the central product of $\phi(G_J)$ and $\phi(G_{I-J})$. 

Note that $\phi(G_J)$ and $\phi(G_{I-J})$ are also finite homomorphic images of products of groups in $\mathcal{C}$. So if both are proper subgroups, then by minimality of $B$, both are in $\mathcal{C}$. Hence by Proposition~\ref{prop:AlmostSimpleCloseCentral}, $B$ as their central product is also in $\mathcal{C}$. A contradiction.

So one of $\phi(G_J)$ and $\phi(G_{I-J})$ is $B$. But since $B$ is the central product of these two groups, the other one must be in the center.
\end{proof}

In particular, let $J=\{i\in I:\text{$G_i$ is solvable}\}$. Then $B$ is the central product of $\phi(G_J)$ and $\phi(G_{I-J})$. but since $B\notin\mathcal{C}$ while $\phi(G_J)$ must be solvable, we see that $B=\phi(G_{I-J})$. So we may safely throw away all solvable factor groups $G_i$ from $G$ when we construct our counterexample. WLOG, from now on we may simply assume that all factor groups $G_i$ of $G$ are almost simple groups.

\begin{prop}[Dichotomy via ultrafilter]
There is an ultrafilter $\omega$ on $I$, such that $\phi(G_J)=B$ if and only if $J\in\omega$, and $\phi(G_J)$ is in the center $\mathrm{Z}(B)$ of $B$ if and only if $J\notin \omega$.
\end{prop}
\begin{proof}
We define $\omega$ to be the collection of all subsets $J$ such that $\phi(G_J)=B$. Then it is enough to show that this is an ultrafilter.

We already see that for any $J$, exactly one of $J$ and $I-J$ is in $\omega$.

If $J_1\subset J_2$ and $J_1\in\omega$, then $G_{J_1}\subseteq G_{J_2}$. Hence $B=\phi(G_{J_1})\subseteq\phi(G_{J_2})$. So $J_2\in\omega$ as desired.

If $J_1,J_2\in\omega$, then $\phi(G_{J_1})=\phi(G_{J_2})=B$. Then, as a result, $\phi(G_{I-J_1}),\phi(G_{I-J_2})$ are both in the center of $B$. Then 
\[
\phi(G_{I-(J_1\cap J_2)})=\phi(G_{(I-J_1)\cup(I- J_2)})=\phi(G_{I-J_1}G_{I-J_2})=\phi(G_{I-J_1})\phi(G_{I-J_2}).
\] 
So $\phi(G_{I-(J_1\cap J_2)})$ is also in the center of $B$. So $\phi(G_{J_1\cap J_2})=B$ and $J_1\cap J_2\in\omega$. 

So $\omega$ is indeed an ultrafilter.
\end{proof}

We fix this ultrafilter $\omega$ from now on. We now show that, in fact, $\CR(B)$ is an $\omega$-ultraproduct of $\CR(G_i)$ for $i\in I$, and the corresponding quotient map is exactly $\phi$ restricted to $\CR(G)$.

\begin{cor}[$\CR(B)$ is simple]
$\phi(\CR(G))=\CR(B)$, and this quotient map yields $\CR(B)$ as an $\omega$-ultraproduct of $\CR(G_i)$ for $i\in I$. In particular, $\CR(B)$ is a non-abelian finite simple group isomorphic to $\CR(G_i)$ for some $i\in I$. 
\end{cor}
\begin{proof}
Note that $\CR(G)$ is a direct product of non-abelian finite simple groups. Hence any finite homomorphic image must also be a direct product of non-abelian finite simple groups. So $\phi(\CR(G))$ is inside of $\CR(B)$.

In particular, $\CR(B)$ cannot be trivial. This is because $B/\CR(B)$ is a finite quotient of $G/\CR(G)$, which is solvable, but $B\notin\mathcal{C}$.

Now $B/\phi(\CR(G))$ is a finite quotient of $G/\CR(G)$, which is a pro-solvable group. So $B/\phi(\CR(G))$ is solvable. In particular, $\CR(B)/\phi(\CR(G))$ is solvable, and thus it must be trivial. So $\phi(\CR(G))=\CR(B)$.

Now for any $J\in\omega$, $\phi(G_J)=B$, and by similar reasoning, we see that $\phi(\CR(G_J))=\CR(B)$. While for $J\notin\omega$, $\phi(\CR(G_J))\subset \phi(G_J)$ is in the center of $B$. But it must also be a direct product of non-abelian finite simple groups. Hence it is trivial.

In particular, the restriction of $\phi$ to $\CR(G)$ must in fact factor through the $\omega$-ultraproduct. Such an ultraproduct must either have no non-trivial finite quotient, or it is isomorphic to one of the factor groups of $\CR(G)$. Since $\CR(B)$ is non-trivial, then $\CR(G)$ has a non-trivial finite quotient. So in our case, $\prod_{\omega} \CR(G_i)$ is isomorphic to the finite simple group $\CR(G_i)$ for some $i\in I$, and $\CR(B)=\phi(\CR(G))$ is a non-trivial quotient of this finite simple group. Hence $\CR(B)$ is a finite simple group.
\end{proof}

Now we establish a description of the center of $B$. Note that $B$ acts on $\CR(B)$ by conjugation, and this induces a natural homomorphism from $B$ to $\Aut(\CR(B))$. Let $N$ be the kernel of this map. Then $B/N$ is by definition almost simple.

\begin{prop}[Description of pre-images of $N$]
For any element $(g_i)_{i\in I}\in G$, we have $\phi((g_i)_{i\in I})\in N$ if and only if $\{i\in I: g_i=e\}\in\omega$.
\end{prop}
\begin{proof}
Suppose we have an element $(g_i)_{i\in I}$ where $\{i\in I: g_i=e\}\in\omega$. Then for any $h_i\in \CR(G_i)$, we have $\{i\in I: g_i^{-1}h_i^{-1}g_ih_i=e\}\in\omega$. Note that, since we picked $(h_i)_{i\in I}\in\CR(G)$, the element $(g_i^{-1}h_i^{-1}g_ih_i)_{i\in I}$ is also in $\CR(G)$, and $\phi$ restricted to $\CR(G)$ is the quotient map onto the $\omega$-ultraproduct. Hence $\phi((g_i^{-1}h_i^{-1}g_ih_i)_{i\in I})=e$. In particular, $\phi((g_i)_{i\in I})$ commutes with $\phi((h_i)_{i\in I})$. Since $\phi(\CR(G))=\CR(B)$, $\phi((h_i)_{i\in I})$ could be any element of $\CR(B)$. Therefore $\phi((g_i)_{i\in I})$ commutes with all elements of $\CR(B)$, and it must be in $N$. So we have proven necessity.

Now we pick any element of $N$, say $\phi((g_i)_{i\in I})\in N$. Note that each $\CR(G_i)$ is generated by at most two elements, say $h_i$ and $h'_i$. Consider the elements $(h_i)_{i\in I}$ and $(h'_i)_{i\in I}$. Their images through $\phi$ are in $\phi(\CR(G))=\CR(B)$, so we know $\phi((g_i)_{i\in I})$ must commute with $\phi((h_i)_{i\in I}),\phi((h_i)_{i\in I})$. In particular, we have $\phi((g_i^{-1}h_i^{-1}g_ih_i)_{i\in I})=\phi((g_i^{-1}{h'_i}^{-1}g_ih'_i)_{i\in I})=e$. However, the element $(g_i^{-1}h_i^{-1}g_ih_i)_{i\in I}$ is in $\CR(G)$, and $\phi$ restricted to this subgroup is the quotient map onto the $\omega$-ultraproduct. Therefore $\phi((g_i^{-1}h_i^{-1}g_ih_i)_{i\in I})=e$ implies that $\{i\in I: g_i^{-1}h_i^{-1}g_ih_i=e\}\in\omega$. Similarly, $\{i\in I: g_i^{-1}h_i^{-1}g_ih_i=e\}\in\omega$ as well. Taking their intersection, we see that $\{i\in I: \text{$g_i$ commutes with $h_i,h'_i$}\}\in\omega$.

But if $g_i$ commutes with $h_i,h'_i$, this means $g_i$ commutes with all elements of $\CR(G_i)$. But since $G_i$ is an almost simple group, this imples that $g_i=e$. Hence $\{i\in I: g_i=e\}\in\omega$.
\end{proof}

\begin{cor}
$N$ is the center of $B$.
\end{cor}
\begin{proof}
Let $Z$ be the center of $B$. Since $B/N$ is almost simple, which is centerless, we must have $Z\subseteq N$.

Now pick any element of $N$, say $\phi((g_i)_{i\in I})\in N$. Then $\{i\in I: g_i=e\}\in\omega$. Let $J=\{i\in I: g_i=e\}$, then $(g_i)_{i\in I}\in G_{I-J}$, where $I-J\notin\omega$. Hence $\phi((g_i)_{i\in I})\in\phi(G_{I-J})$ is in the center of $B$. So $N\subseteq Z$.
\end{proof}

Now we proceed to show that $N$ is in fact a direct factor of $B$, which would result in a contradiction to the minimality of $B$, finishing our proof of Proposition~\ref{prop:AlmostSimpleProjectiveLimit}.

\begin{prop}
The derived subgroup $B'$ has a trivial intersection with $N$.
\end{prop}
\begin{proof}
Pick any element of $B'$, say $\phi(([a_i,b_i])_{i\in I})$. Suppose this is also inside $N$, then $\{i\in I: [a_i,b_i]=e\}\in\omega$.

Let $J=\{i\in I: [a_i,b_i]=e\}$. We define $a'_i=a_i$ if $i\notin J$ and $a'_i=e$ if $i\in J$. Similarly, we define $b'_i=b_i$ if $i\notin J$ and $b'_i=e$ if $i\in J$. Then $[a'_i,b'_i]=[a_i,b_i]$ always. So $\phi(([a_i,b_i])_{i\in I})=\phi(([a'_i,b'_i])_{i\in I})$.

However, we clearly have $(a'_i)_{i\in I},(b'_i)_{i\in I}\in G_{I-J}$. So $\phi((a'_i)_{i\in I}),\phi((b'_i)_{i\in I})$ are in the center $N$. So we have 
\[
\phi(([a'_i,b'_i])_{i\in I})=[\phi((a'_i)_{i\in I}),\phi((b'_i)_{i\in I})]=e.
\] 
Hence $B'$ has a trivial intersection with $N$.
\end{proof}

For the moment, we use the notation $B^k:=\{g^k\mid g\in B\}$ and $N^k:=\{g^k\mid g\in N\}$.

\begin{prop}
For any integer $k$, $NB'\cap B^kB'\subseteq N^kB'$.
\end{prop}
\begin{proof}
First let us show that $NB'\cap B^k\subseteq N^kB'$. Suppose $\phi((x_i)_{i\in I}^k)[\phi((a_i)_{i\in I}),\phi((b_i)_{i\in I})]\in NB'\cap B^k$, say 
\[
\phi((x_i)_{i\in I}^k)=[\phi((a_i)_{i\in I}),\phi((b_i)_{i\in I})]\phi((y_i)_{i\in I}).
\] 
Here $y_i=x_i[a_i,b_i]^{-1}$ and $\phi((y_i)_{i\in I})\in N$. Then since $\phi((y_i)_{i\in I})\in N$, the index set $J=\{i\in I:y_i=e\}$ is an element of $\omega$.

Define $a'_i,b'_i,x'_i$ to be $a_i,b_i,x_i$ respectively if $i\notin J$, and to be $e$ if $i\in J$. Then we have
\[
\phi((x'_i)_{i\in I}^k)=[\phi((a'_i)_{i\in I}),\phi((b'_i)_{i\in I})]\phi((y_i)_{i\in I}).
\] 
But $\phi((a'_i)_{i\in I}),\phi((b'_i)_{i\in I})\in N$ since $\omega$-almost-all coordinates of $(a'_i)_{i\in I}$ and $(b'_i)_{i\in I}$ are trivial. Hence their commutator is also trivial. So $\phi((y_i)_{i\in I})=\phi((x'_i)_{i\in I}^k)$. We also have $\phi((x'_i)_{i\in I})\in N$ since $\omega$-almost-all coordinates of $(x'_i)_{i\in I}$ are trivial. Therefore $\phi((y_i)_{i\in I})\in N^k$.

So we have 
\[
\phi((x_i)_{i\in I}^k)=[\phi((a_i)_{i\in I}),\phi((b_i)_{i\in I})]\phi((y_i)_{i\in I})\in N^kB'.
\]

Now, suppose $gh\in NB'\cap B^kB'$ where $g\in B^k$ and $h\in B'$. Then $g\in NB' h^{-1}=NB'$. So we have 
\[
g\in NB'\cap B^k\subseteq N^kB'.
\] 
Hence $gh\in N^kB'B'=N^kB'$.
\end{proof}

\begin{cor}
\label{cor:BdoesNotExist}
$B$ cannot exist.
\end{cor}
\begin{proof}
We claim that $N$ is a direct factor of $B$. Consider $NB'/B'$ in the abelian group $B/B'$. Then for any integer $k$, 
\[
(NB'/B')\cap(B/B')^k\subseteq(NB'\cap B^kB')/B'\subseteq N^kB'/B'=(NB'/B')^k.
\] 

By the Kulikov criteria \cite{Kulikov45}, this means that $NB'/B'$ is a direct summand of the abelian group $B/B'$. Let its complement subgroup be $H/B'$ for some subgroup $H$ of $B$ containing $B'$. So $NB'\cap H=B'$ and $(NB')H=B$.

Then 
\[
N\cap H\subseteq N\cap NB'\cap H=N\cap B'.
\]
So $N\cap H$ is trivial. But we also have $B=(NB')H=NH$. Hence $B$ is the internal direct product of $N$ and $H$.

But now $N$ is abelian, and $H\simeq B/N$ is almost simple. So $B\in\mathcal{C}$, a contradiction.
\end{proof}

\subsection{Almost semisimple groups}
\label{subsec:almostsemisimple}

When studying a class of groups, it is sometimes useful if the class of groups is in fact a formation of groups.

\begin{defn}
A class of groups is a formation if it is closed under quotients and subdirect products.
\end{defn}

Note that the class of finite products of finite almost simple groups and finite solvable groups is \emph{not} closed under subdirect product. Hence it is not a formation of groups.

\begin{eg}
Think of the symmetric group $\mathrm{S}_5$ as the semi-direct product $\A_5\rtimes\Z/2\Z$, and the symmetric group $\mathrm{S}_6$ as the semi-direct product $\A_6\rtimes\Z/2\Z$. Now, let $\Z/2\Z$ act on both $\A_5$ and $\A_6$ simultaneously, then we have a group $(\A_5\times\A_6)\rtimes\Z/2\Z$, which is a sub-direct product of the almost simple groups $\mathrm{S}_5$ and $\mathrm{S}_6$. However, the resulting group is \emph{not} a direct product of almost simple groups and solvable groups.
%Think of the symmetric group $\mathrm{S}_5$ as the semi-direct product $\A_5\rtimes\Z/2\Z$, and the solvable symmetric group $\mathrm{S}_3$ as the semi-direct product $\Z/3\Z\rtimes\Z/2\Z$. Now, let $\Z/2\Z$ acts on both $\A_5$ and $\Z/3\Z$ simultaneously, then we have a group $(\A_5\times\Z/3\Z)\rtimes\Z/2\Z$, which is a sub-direct product of the almost simple group $\mathrm{S}_5$ and the solvable group $\mathrm{S}_3$. However, the resulting group is NOT a direct product of almost simple groups and solvable groups.
\end{eg}

So if we want to study a formation of groups, we need to consider a larger class of groups.

\begin{defn}
We define a group $G$ to be almost semisimple if $G$ has no solvable normal subgroups.
\end{defn}

These groups are natural generalizations of almost simple groups.

\begin{prop}
\label{prop:AlmostSemiFaith}
If $G$ is an almost semisimple group, then $G$ acts on its CR-radical $\CR(G)$ faithfully. In particular, we may think of $G$ as a subgroup of $\Aut(S)$ containing $S$ for some semisimple group $S$.
\end{prop}
\begin{proof}
Note that $G$ acts on its CR-radical $\CR(G)$ by conjugation. Let $N$ be the kernel of the induced homomorphism $G\to\mathrm{Aut}(\CR(G))$. Suppose for contradiction that it is not trivial. Let $M$ be a minimal non-trivial normal subgroup of $N$. Then $M$ is either semisimple or abelian.

If $M$ is semisimple, then $M\subseteq\CR(N)$. Since $\CR(N)$ is characteristic in $N$, it is also a semisimple normal subgroup of $G$, so $\CR(N)\subseteq\CR(G)$ and $\CR(N)\subseteq N$. However, since $\CR(G)$ is semisimple, therefore it acts on itself faithfully. So the intersection $N\cap\CR(G)$ is trivial, which implies that $\CR(N)$ is trivial as well. Then $\CR(N)$ cannot contain $M$, a contradiction. 

Suppose $M$ is abelian. Then the unique maximal solvable normal subgroup of $N$ is non-trivial. But this subgroup is also characteristic in $N$, hence it is a non-trivial solvable normal subgroup of $G$, which is impossible as we require $G$ to be almost semisimple. Contradiction.

So $N$ must be trivial.
\end{proof}

Now to have a formation of groups, we also want our class to be closed under quotient. Hence we need to throw solvable groups into the mix as well.

\begin{prop}
If $G/\CR(G)$ is solvable, then $G$ is a subdirect product of a solvable group and an almost semisimple group.
\end{prop}
\begin{proof}
Let $\mathrm{sol}(G)$ be the solvable radical of $G$ (i.e., the unique largest solvable normal subgroup of $G$). Then $G/\mathrm{sol}(G)$ is an almost simple group, while $G/\CR(G)$ is solvable. Finally, the intersection of $\mathrm{sol}(G)$ and $\CR(G)$ must be both CR and solvable, so it is trivial. So we are done.
\end{proof}

Now let $\mathcal{C}$ be the class of groups $G$ such that $G/\CR(G)$ is solvable. It is clearly closed under taking quotients.

\begin{prop}
$\mathcal{C}$ is closed under subdirect product.
\end{prop}
\begin{proof}
Suppose $G$ has two normal subgroups $N_1,N_2$ such that $G/N_1,G/N_2\in\mathcal{C}$. We can find normal subgroups $M_1,M_2$ of $G$ that are the pre-images of $\CR(G/N_1),\CR(G/N_2)$. Then $G/M_1,G/M_2$ is solvable. As a result, since $G/(M_1\cap M_2)$ is a subdirect product of solvable groups $G/M_1,G/M_2$, $G/(M_1\cap M_2)$ is also solvable, and its subgroups $M_1/(M_1\cap M_2)$ and $M_2/(M_1\cap M_2)$ are also solvable.

\begin{figure}[h!]
\centering
\begin{tikzpicture}[description/.style={fill=white,inner sep=2pt}]
\matrix (m) [matrix of math nodes, row sep=1.5em,
column sep=0.3em, text height=1.5ex, text depth=0.25ex]
{ 
&&G&&\\
N_1(M_1\cap M_2)&M_1&&M_2&N_2(M_1\cap M_2)\\
N_1&&M_1\cap M_2&&N_2\\
&N_1\cap M_2&&N_2\cap M_1&\\
&&\{e\}&&\\};

\path[-] (m-1-3) edge node [above,sloped]{\tiny solvable}  (m-2-2)
	(m-1-3) edge node [above,sloped]{\tiny solvable} (m-2-4)
	(m-2-1) edge [double equal sign distance] (m-2-2)
	(m-2-4) edge [double equal sign distance] (m-2-5)
	(m-2-2) edge node [below,sloped]{\tiny CR} (m-3-1)
	(m-2-2) edge node [above,sloped]{\tiny solvable} (m-3-3)
	(m-2-4) edge node [below,sloped]{\tiny CR} (m-3-5)
	(m-2-4) edge node [above,sloped]{\tiny solvable} (m-3-3)
	(m-3-1) edge (m-4-2)
	(m-3-3) edge node [above,sloped]{\tiny CR} (m-4-2)
	(m-3-5) edge (m-4-4)
	(m-3-3) edge node [above,sloped]{\tiny CR} (m-4-4)
	(m-4-2) edge (m-5-3)
	(m-4-4) edge (m-5-3);
%
%\path[-] (m-1-1) edge (m-3-1)
%		 (m-3-1) edge  (m-4-3)
%		         edge [-,line width=1pt, draw=blue](m-5-1)
%		 (m-5-1) edge (m-6-3)
%		         edge (m-7-1)
%		 (m-7-1) edge (m-8-2)
%		 (m-8-2) edge (m-6-3)
%		 (m-6-3) edge (m-8-4)
%		         edge (m-5-5)
%		         edge [-,line width=1pt, draw=blue](m-4-3)
%		 (m-7-5) edge (m-5-5)
%		 		 edge (m-8-4)
%		 (m-1-5) edge (m-3-5)
%		 (m-3-5) edge (m-4-3)
%		 		 edge [-,line width=1pt, draw=blue](m-5-5);
\end{tikzpicture}
%\caption{The Butterfly}
\end{figure}

Let us now show that $N_1(M_1\cap M_2)=M_1$. To see this, note that $M_1/(N_1(M_1\cap M_2))$ is a quotient of the solvable group $M_1/(M_1\cap M_2)$, hence it is solvable. However, it is also the quotient of $M_1/N_1\simeq \CR(G/N_1)$, therefore it is a CR group. Since it is both solvable and CR, it has to be trivial. So we have $N_1(M_1\cap M_2)=M_1$, and similarly $N_2(M_1\cap M_2)=M_2$ as well.

So we have 
\[(M_1\cap M_2)/(N_1\cap M_2)=(M_1\cap M_2)/(N_1\cap (M_1\cap M_2))\simeq N_1(M_1\cap M_2)/N_1=M_1/N_1\simeq \CR(G/N_1).\]

So $(M_1\cap M_2)/(N_1\cap M_2)$ and similarly $(M_1\cap M_2)/(M_1\cap N_2)$ are both CR groups. We also see that $(N_1\cap M_2)\cap(M_1\cap N_2)\subseteq N_1\cap N_2$ is trivial. Hence $M_1\cap M_2$ is the subdirect product of CR groups  $(M_1\cap M_2)/(N_1\cap M_2)$ and $(M_1\cap M_2)/(M_1\cap N_2)$. So $M_1\cap M_2$ must be a CR group itself. In particular, $M_1\cap M_2\subseteq\CR(G)$.

On the other hand, since $G/(M_1\cap M_2)$ is solvable, $\CR(G)\subseteq M_1\cap M_2$. So we see that $M_1\cap M_2=\CR(G)$, and $G/\CR(G)$ is solvable.
\end{proof}

\begin{prop}
If $G$ is a pro-$\mathcal{C}$ group, then any finite homomorphic image of $G$ must still be in $\mathcal{C}$.
\end{prop}
\begin{proof}
Suppose $G$ is formed via a surjective inverse directed system $\{G_i\}_{i\in I}$.  Let $S_i$ be the CR-radical of $G_i$.

For any surjective homomorphism $\phi_{ij}:G_i\to G_j$, then $S_i$ must be mapped into $S_j$. Hence this induces homomorphisms $S_i\to S_j$ and $G_i/S_i\to G_j/S_j$. So the inverse directed system $\{G_i\}_{i\in I}$ would induce the corresponding inverse directed systems $\{S_i\}_{i\in I}$ and $\{G_i/S_i\}_{i\in I}$. 

Let $S$ be the projective limit of the system $\{S_i\}_{i\in I}$, then $S$ is a semisimple profinite group and a closed normal subgroup of $G$. Note that the projective limit on inverse systems of finite groups is an exact functor. (See, e.g., Proposition 2.2.4 of \cite{RZ00}.) As a result, $G/S$ is the projective limit of the system $\{G_i/S_i\}_{i\in I}$, and therefore it is pro-solvable.

Now suppose $B$ is a finite quotient of $G$. Let $q$ be the quotient map from $G$ to $B$. Then since $S$ is a semisimple profinite group, $q(S)$ must be a semisimple normal subgroup of $B$. Hence $q(S)\leq \CR(B)$. So $S$ is in the kernel of the quotient map from $G$ to $B/\CR(B)$. This induces a surjective homomorphism from $G/S$ to $B/\CR(B)$. Since $G/S$ is pro-solvable, we see that $B/\CR(B)$ is solvable. Hence $B$ is in $\mathcal{C}$.
\end{proof}

\subsection{Perfect groups with bounded commutator width}
\label{subsec:bddcommwidth}

We devote a small section to perfect groups with bounded commutator width.

\begin{prop}
Let $P_n$ be a surjective inverse direct system of perfect groups, and suppose they all have commutator width at most $d$. Let $G$ be their projective limit. Then $G$ is also perfect with commutator width $d$.
\end{prop}

\begin{proof}
Since the class of finite perfect groups with commutator width at most $d$ is quotient closed, therefore any continuous finite quotient of $G$ must be perfect with commutator width at most $d$. Let $X$ be the set of commutators in $G$, then $X^{\star d}$ is mapped surjectively onto all continuous finite quotients. In particular, $X^{\star d}$ is dense in $G$.

However, consider the map $w:G^{2d}\to G$ sending $(g_1,h_1,\dots,g_d,h_d)$ to the element $[g_1,h_1]\dots[g_d,h_d]$. Since $w$ must be continuous, and $G^{2d}$ is compact, therefore its image $X^{\star d}=w(G^{2d})$ is compact. Since $G$ is Hausdorff, $X^{\star d}$ is a dense closed subset of $G$. Hence $X^{\star d}=G$.

So $G$ is perfect with commutator length at most $d$.
\end{proof}

\section{Perfect Groups with Unbounded Commutator Width}
\label{sec:perfect}

In this section, we shall prove Theorem~\ref{thm:perfect} according to the layout in Section~\ref{subsec:layout}.

\subsection{The restricted Burnside coproduct}
\label{subsec:RestrictedBurnside}

For any positive integer $m$, consider the category $\mathcal{C}$ of finite groups with exponent dividing $m$. Then Theorem~\ref{thm:Zelmanov} showed that $B(d,m)$ is the free object on $d$ generators in this category.

\begin{prop}
The category $\mathcal{C}$ is a reflexive subcategory of the category of finitely generated groups (i.e., the inclusion functor has a left adjoint).
\end{prop}
\begin{proof}
Let us create this reflexive functor to the inclusion functor. Given any finitely generated group $G$, then it is the quotient of the free group $\mathrm{F}_d$ for some positive integer $d$. Let the kernel of this quotient map $q:\mathrm{F}_d\to G$ be $N_G$.

According to Theorem~\ref{thm:Zelmanov}, the free group $\mathrm{F}_d$ has a normal subgroup $N_B$ such that any group homomorphism from $\mathrm{F}_d$ to a group in $\mathcal{C}$ must factor through $\mathrm{F}_d/N_B$. Let $N_m$ be the normal subgroup generated by $N_G$ and $N_B$, and we define $G_m=F/N_m$. Since $G_m$ is a quotient of the restricted Burnside group $B(d,m)$, therefore $G_m\in\mathcal{C}$ as desired.

We claim that for any $H\in\mathcal{C}$, any homomorphism $f:G\to H$, then there is a unique induced homomorphism $f_m:G_m\to H$. This would show that the functor $G\mapsto G_m$ is indeed a left adjoint to the inclusion functor.

To see the claim, note that $f\circ q:\mathrm{F}_d\to H$ must have both $N_G$ and $N_B$ in its kernel. So $N_m$ is in this kernel. Hence this map factors through a unique homomorphism $f_m:G_m\to H$.
\end{proof}

\begin{cor}
In the category $\mathcal{C}$, finite product and finite coproduct exist. And the coproduct of $B(d_1,m)$ and $B(d_2,m)$ is isomorphic to $B(d_1+d_2,m)$.
\end{cor}
\begin{proof}
Since the category of finitely generated groups has coproducts, therefore any reflexive subcategory has coproducts. The other statements are trivial.
\end{proof}

We define a functor $\mathcal{B}_d:\mathcal{C}\to\mathcal{C}$ that maps each group $G$ in $\mathcal{C}$ to the coproduct of $G$ with itself $d$-times, and maps each morphism $\phi:G\to H$ to the corresponding morphism on the $d$-times iterated coproducts.

\begin{prop}
$\mathcal{B}_d$ preserves all colimits.
\end{prop}
\begin{proof}
Note that the coproduct functor $\mathcal{F}:\mathcal{C}^d\to \mathcal{C}$  is left adjoint to the diagonal functor $\mathcal{G}:\mathcal{C}\to\mathcal{C}^d$, while the diagonal functor $\mathcal{G}$  is left adjoint to the product functor. Therefore the functor $\mathcal{B}_d=\mathcal{F}\circ\mathcal{G}$ is left adjoint and preserves colimits.
\end{proof}

\begin{cor}
$\mathcal{B}_d(\phi)$ is surjective if $\phi$ is surjective.
\end{cor}

\begin{lem}
Let $\iota_1,\dots,\iota_d:G\to \mathcal{B}_d(G)$ be the coprojection morphisms. Then they are injective. Furthermore, $\iota_1(G),\dots,\iota_d(G)$ are independent subgroups generating $\mathcal{B}_d(G)$.
\end{lem}
\begin{proof}
Since the category $\mathcal{C}$ has zero morphisms, therefore all coprojections are split monomorphisms, and the images $\iota_1(G),\dots,\iota_d(G)$ are independent subgroups.

Since the restricted Burnside coproduct $\mathcal{B}_d(G)$ is a quotient of the free product of $d$ copies of $G$, therefore the subgroups $\iota_1(G),\dots,\iota_d(G)$ generates $\mathcal{B}_d(G)$.
%
%Let us first show that these canonical homomorphisms are injective. Fix an index $i$, and define homomorphisms $\phi_1,\dots,\phi_d:G\to G$ such that $\phi_i$ is the identity homomorphism, while all the others are trivial homomorphisms. Then by definition of coproduct, this induces a map $\phi:\mathcal{B}_d(G)\to G$ such that $\phi\circ\iota_i=\phi_i$ is the identity homomorphism. So $\iota_i$ is injective. We also see that $\phi(\iota_j(G))=\phi_j(G)$ is trivial for all $j\neq i$. So the subgroup generated by $\iota_j(G)$ for all $j\neq i$ is independent from $\iota_i(G)$. So the subgroups $\iota_1(G),\dots,\iota_d(G)$ are independent.
%
%WLOG, suppose $G=B(k,m)/N$ for a normal subgroup $N$ of $B(k,m)$. Let $q$ be the quotient homomorphism. Then $\mathcal{B}_d(q)$ is a surjective homomorphism from $\mathcal{B}_d(B(k,m))$ to $\mathcal{B}_d(G)$.
%
%By a slight abuse of notation, let us also use the notation $\iota_1,\dots,\iota_d:B(k,m)\to \mathcal{B}_d(B(k,m))$ for the canonical homomorphisms. Then $\mathcal{B}_d(q)\circ\iota_i=\iota_i\circ q$. 
%
%Since $\mathcal{B}_d(B(k,m))\simeq B(kd,m)$ is generated by $\iota_i(B(k,m))$, by surjectivity of $\mathcal{B}_d(q)$, we see that $\mathcal{B}_d(G)$ is generated by $\mathcal{B}_d(q)\circ\iota_i(B(k,m))=\iota_i\circ q(B(k,m))=\iota_i(G)$.
\end{proof}

We now examine the interaction of $\mathcal{B}_d$ with group actions. In the following, $M$ is a finite group with exponents dividing $m$, and $M_1,\dots,M_d$ are the images of the coprojection morphisms from $M$ into $\mathcal{B}_d(M)$. We use $e_1,\dots,e_d$ to denote the generators of $B(d,m)=\mathcal{B}_d(\Z/m\Z)$.

\begin{prop}
\label{prop:BdAction}
Suppose $G$ is a group acting on a finite group $M$, and the exponent of $M$ divides $m$. Then $G$ has an induced action on $\mathcal{B}_d(M)$. 

If the coprojection images of $M$ in $\mathcal{B}_d(M)$ are $M_1,\dots,M_d$, then these subgroups are $G$-invariant, and the $G$-action on $\mathcal{B}_d(M)$ restricted to each $M_i$ is identical to the $G$-action on $M$.
\end{prop}
\begin{proof}
Suppose we have $\phi:G\to\Aut(M)$. As $\mathcal{B}_d$ is functorial and $M$ is a finite group with exponent dividing $m$, we have a map $\mathcal{B}_d:\Aut(M)\to\Aut(\mathcal{B}_d(M))$. So $\mathcal{B}_d\circ\phi$ is the induced $G$-action on $\mathcal{B}_d(M)$.

Let $\iota_1,\dots,\iota_d:M\to \mathcal{B}_d(M)$ be the coprojection morphisms so that $M_i=\iota_i(M)$. For any $\phi\in\Aut(M)$, then $\iota_i\circ\phi=\mathcal{B}_d(\phi)\circ\iota_i$ by definition of the coproduct. So 
\[
\mathcal{B}_d(\phi)(M_i)=
\mathcal{B}_d(\phi)\circ\iota_i(M)=
\iota_i\circ\phi(M)=\iota_i(M)=M_i.
\]

So $M_i$ is $\mathcal{B}_d(\phi)$-invariant, and the $\mathcal{B}_d(\phi)$-action on $\mathcal{B}_d(M)$ restricted to $M_i$ is identical to the $\phi$-action on $M$.
\end{proof}

\begin{prop}
\label{prop:BdActionPerfect}
Suppose $G$ is a group acting on a finite group $M$, and the exponent of $M$ divides $m$, and $M\rtimes G$ is perfect. Then for the induced action of $G$ on $\mathcal{B}_d(M)$, and $\mathcal{B}_d(M)\rtimes G$ is perfect.
\end{prop}
\begin{proof}
Now $\mathcal{B}_d(M)$ is generated by $G$-invariant subgroups $M_1,\dots,M_d$ as in Proposition~\ref{prop:BdAction}. However, since $G$ acts on $M_i$ the same way as its action on $M$, therefore $M_i\rtimes G$ as a subgroup of $\mathcal{B}_d(M)\rtimes G$ is isomorphic to $M\rtimes G$, which is perfect. As these perfect subgroups $M_i\rtimes G$ generate the whole group, $\mathcal{B}_d(M)\rtimes G$ is perfect.
\end{proof}

\begin{prop}
\label{prop:BdActionKernel}
Suppose $G$ is a group acting on a finite group $M$, and the exponent of $M$ divides $m$. Then for any $G$-invariant homomorphism $\phi:M\to\Z/m\Z$, the induced homomorphism $\mathcal{B}_d(\phi)$ is $G$-invariant.
\end{prop}
\begin{proof}
Since $\phi$ is $G$-invariant, for any $x\in M$ and $g\in G$, we have $\phi(x)=\phi(x^g)$. Let $\iota_1,\dots,\iota_d:M\to \mathcal{B}_d(M)$ be the coprojection morphisms so that $M_i=\iota_i(M)$. Then 
\[
\mathcal{B}_d(\phi)\circ\iota_i(x^g)=\iota_i(\phi(x^g))=\iota_i(\phi(x))=\mathcal{B}_d(\phi)\circ\iota_i(x).
\] 
So $\mathcal{B}_d(\phi)$ restricted to each $M_i$ is $G$-invariant. But since these $M_i$ generates the domain $\mathcal{B}_d(M)$, therefore we see that $\mathcal{B}_d(\phi)$ is $G$-invariant.
\end{proof}

\subsection{The ultraproduct functor}
\label{subsec:UltraFunctor}

The ultraproduct of groups can be thought of as a functor from $\mathrm{Grp}^\N$ to $\mathrm{Grp}$. It has the following categorical properties:

\begin{enumerate}
\item The ultraproduct functor preserves all finite limits. (In particular, the ultraproduct of subgroups is a subgroup of the ultraproduct, and the ultraproduct of intersections is the intersection of ultraproducts.)
\item The ultraproduct functor preserves short exact sequences. (In particular, it preserves normal subgroups, monomorphism, epimorphism, and quotients.)
\item The ultraproduct functor preserves semi-direct products.
\item The ultraproduct of abelian groups is an abelian group.
\item The ultraproduct of groups of exponent dividing $m$ is a group of exponent dividing $m$.
\item The ultraproduct of a sequence of the same finite group is the finite group itself.
\end{enumerate}

%Given a sequence of groups $\{G_n\}_{n\in\N}$, we can define their ultraproduct to be $\prod_{n\to\omega}G_n:=(\prod_{n\in\N}G_n)/N_\omega$, where $N_{\omega}=\{(g_n)_{n\in\N}\in\prod_{n\in\N}G_n:\{n\in N:g_n=e\}\in\omega\}$. Here $e$ is the identity element as usual. For any element $(g_n)_{n\in\N}\in\prod_{n\in\N}G_n$, we call its image after the quotient by $N_\omega$ the ultralimit of the sequence, and denote it as $\lim_{n\to\omega}g_n$. This is the ultralimit of the sequence $(g_n)_{n\in\N}$. Note that $\lim_{n\to\omega}g_n=\lim_{n\to\omega}h_n$ if and only if $\{n\in\N:g_n=h_n\}\in\omega$. Intuitively, if we think of sets in $\omega$ as ``large'' and sets not in $\omega$ as ``small'', then the ultraproduct/ultralimit construction means we are allowed to ignore a ``small'' set of coordinates. And two elements in the ultraproduct are the same if they agree on a ``large'' set of coordinates.
%
%In particular, taking ultraproduct is functorial. The ultraproduct functor from $\mathrm{Grp}^\N$ to $\mathrm{Grp}^\N$ is made via direct products and a filtered colimit. So it preserves all finite limit.
%
%One last nice thing about the ultraproduct is the fact that $\omega$ is a maximal filter. 

Given any sequence of finite groups $M_n$ of exponents dividing a fixed integer $m$, we want to understand homomorphisms from $\prod_{n\to\omega}M_n$ to $\Z/m\Z$. Note that many of such homomorphisms are elements of $\prod_{n\to\omega}M_n^*$, as shown in the example below.

\begin{eg}
One way to construct such a homomorphism is to first pick a homomorphism $\phi_n:M_n\to\Z/m\Z$ for each $n$, and take the ultralimit $\lim_{n\to\omega}\phi_n$, which is a homomorphism from $\prod_{n\to\omega}M_n$ to $\prod_{n\to\omega}\Z/m\Z=\Z/m\Z$. 

If we use $M_n^*$ to denote the abelian group of homomorphisms from $M_n$ to $\Z/m\Z$ (i.e., the ``dual group''), then these $\lim_{n\to\omega}\phi_n$ correspond to elements of the ultraproduct $\prod_{n\to\omega}M_n^*$.
\end{eg}

Unfortunately, not all homomorphisms from $\prod_{n\to\omega}M_n$ to $\Z/m\Z$ are elements of $\prod_{n\to\omega}M_n^*$. Here is another kind of construction: we take all homomorphisms from $\prod_{n\to\omega}M_n$ to $\Z/m\Z$ in $\prod_{n\to\omega}M_n^*$ as above, and take their ``consensus'' according to some ultrafilter $\U$ on $\prod_{n\to\omega}M_n^*$.

\begin{defn}
Let $\U$ be any ultrafilter on the set $\prod_{n\to\omega}M_n^*$, then we define the $\U$-consensus to be the homomorphism $\phi_\U=\lim_{\phi\to\U}\phi$ such that $\phi_\U(x)=\lim_{\phi\to\U}\phi(x)$ for all $x\in \prod_{n\to\omega}M_n$,
\end{defn}

Then it is straight forward to verify that $\phi_\U:\prod_{n\to\omega}M_n\to\Z/m\Z$ is indeed a homomorphism.

It turns out that, when all groups $M_n$ are abelian, then any homomorphism $\phi:\prod_{n\to\omega}M_n\to\Z/m\Z$ is equal to $\phi_\U$ for some ultrafilter $\mathcal{U}$ on $\prod_{n\to\omega}M_n^*$, as in the above example.

\begin{prop}
Fix a sequence of finite abelian groups $M_n$ of exponents dividing a fixed integer $m$. Then any homomorphism $\phi:\prod_{n\to\omega}M_n\to\Z/m\Z$ is equal to $\phi_{\mathcal{U}}$ for some ultrafilter $\mathcal{U}$ on $\prod_{n\to\omega}M_n^*$.
\end{prop}
\begin{proof}
This is a special case of Proposition~\ref{prop:AnyToU} below, when the group actions are all trivial.
\end{proof}

In fact, we can prove something stronger. Proposition~\ref{prop:AnyToU} below shows that we can in fact do this even with respect to a given group action. To do this, we need a definition of $G$-invariant ultrafilter on $\prod_{n\to\omega}M_n^*$.

Suppose we have a finite group $G_n$ acting on $M_n$ for each $n$. Then there is an induced action of $G=\prod_{n\to\omega}G_n$ on $\prod_{n\to\omega}M_n$. 

For each $g\in G$, let $S_g\subseteq\prod_{n\to\omega}M_n^*$ be the collection of all $g$-invariant elements of $\prod_{n\to\omega}M_n^*$. Then we say the ultrafilter $\U$ is $G$-invariant if and only if $S_g\in\U$ for all $g\in G$. Then we have the following result.

\begin{prop}
\label{prop:AnyToU}
Fix a sequence of finite abelian groups $M_n$ of exponents dividing a fixed integer $m$. Suppose for each $n$, we have a finite group $G_n$ acting on $M_n$. So we have an action of $G=\prod_{n\to\omega}G_n$ on $\prod_{n\to\omega}M_n$.

Then any $G$-invariant homomorphism $\phi:\prod_{n\to\omega}M_n\to\Z/m\Z$ is equal to $\phi_{\mathcal{U}}$ for some $G$-invariant ultrafilter $\mathcal{U}$ on $\prod_{n\to\omega}M_n^*$.
\end{prop}
\begin{proof}
For each $g\in G$, let $S_g\subseteq\prod_{n\to\omega}M_n^*$ be the collection of all $g$-invariant elements of $\prod_{n\to\omega}M_n^*$. For each $x\in\prod_{n\to\omega}M_n$, let $S_x\subseteq\prod_{n\to\omega}M_n^*$ be the collection of all $\psi\in\prod_{n\to\omega}M_n^*$ such that $\psi(x)=\phi(x)$.

Then $\U$ is $G$-invariant if and only if it contains all $S_g$ for all $g\in G$, and we have $\phi=\phi_\U$ if and only if $\U$ contains $S_x$ for all $x\in\prod_{n\to\omega}M_n$. So our proposition is the same as the claim that there is an ultrafilter on $\prod_{n\to\omega}M_n^*$ containing all $S_g$ and all $S_x$.

To prove this claim, we need to show that the collection of all $S_g$ and all $S_x$ has finite intersection property. This is done by Lemma~\ref{lem:SgSxHasFIP}.
\end{proof}

First, let us show that the collection of all $S_x$ has the finite intersection property.

\begin{lem}
\label{lem:SxHasFIP}
Let $M_n$ be a sequence of finite abelian groups with exponents dividing $m$. For any $x_1,\dots,x_k\in\prod_{n\to\omega}M_n$ and any homomorphism $\phi:\prod_{n\to\omega}M_n\to\Z/m\Z$, we can find a homomorphism $\psi\in\prod_{n\to\omega}M_n^*$ such that $\psi(x_i)=\phi(x_i)$ for all $i$.
\end{lem}
\begin{proof}
Note that finite abelian groups with exponents dividing $m$ are the same as $\Z/m\Z$-modules. Picking elements $x_1,\dots,x_k\in\prod_{n\to\omega}M_n$ is the same as picking a $\Z/m\Z$-modules homomorphism $f:(\Z/m\Z)^k\to\prod_{n\to\omega}M_n$. Say we let $e_1,\dots,e_k$ be the standard generators of $\Z/m\Z$, so that we have $f(e_i)=x_i$.

Let $x_i=\lim_{n\to\omega}x_{i,n}$, then we can define $\Z/m\Z$-modules homomorphism $f_n:(\Z/m\Z)^k\to M_n$ such that $f_n(e_i)=x_{i,n}$. Then $f=\lim_{n\to\omega}f_n$ in the sense that $f(v)=\lim_{n\to\omega}f_n(v)$ for all $v\in(\Z/m\Z)^k$.

Let $N=\Ker(f)$. Since $f(v)=0$ if and only if $v\in N$, the sets $\{n\in\N\mid f_n(v)=0\}$ for $v\in N$ are inside the ultrafilter $\omega$, and the sets $\{n\in\N\mid f_n(v)\neq 0\}$ for $v\in (\Z/m\Z)^k-N$ are also inside the ultrafilter $\omega$. Since $(\Z/m\Z)^k$ is a finite set, we have the following result. 
\[
\{n\in\N\mid \Ker(f_n)=N\}=(\bigcap_{v\in N}\{n\in\N\mid f_n(v)=0\})\cap(\bigcap_{v\in (\Z/m\Z)^k-N}\{n\in\N\mid f_n(v)\neq 0\})\in\omega.
\]

So $\Ker(f_n)=N$ for $\omega$-almost all $n$.

\begin{figure}[h!]
\centering
\begin{tikzcd}
(\Z/m\Z)^k\arrow{r}{q}\arrow[swap]{dr}{\phi\circ f}\arrow[bend left]{rr}{f_n} 
&(\Z/m\Z)^k/N \arrow{r}{f'_n}\arrow{d}{f'} &M_n\arrow[dashed]{dl}{\psi_n}\\
&\Z/m\Z&
\end{tikzcd}
\end{figure}

Now the homomorphism $\phi\circ f:(\Z/m\Z)^k\to\Z/m\Z$ factors through the quotient map $q:(\Z/m\Z)^k\to (\Z/m\Z)^k/N$, and induces a homomorphism $f':(\Z/m\Z)^k/N\to\Z/m\Z$. For $\omega$-almost all $n$, the map $f_n:(\Z/m\Z)^k\to M_n$ also factors through $q$ as well, and induces an injective homomorphism $f'_m:(\Z/m\Z)^k/N\to M_n$. 

Note that by Baer's criterion (see, e.g., \cite{Lam2012}), $\Z/m\Z$ is an injective module over itself. Since for $\omega$-almost all $n$, $f'_n:(\Z/m\Z)^k/N\to M_n$ is injective, therefore we have an extension $\psi_n:M_n\to\Z/m\Z$ such that $\psi_n\circ f'_n=f'$. Clearly $\psi_n\in M_n^*$. In particular, $\psi_n\circ f_n=\phi\circ f$ for $\omega$-almost all $n$.

Let $\psi=\lim_{n\to\omega} \psi_n\in\prod_{n\to\omega}M_n^*$. Let us verify that $\psi$ is what we need.
\[
\psi(x_i)=\lim_{n\to\omega} \psi_n(x_{i,n})=\lim_{n\to\omega} \psi_n(f_n(e_i))=\lim_{n\to\omega}  \phi\circ f(e_i)=\phi\circ f(e_i)= \phi(x_i).
\]

So we are done.
\end{proof}

\begin{lem}
\label{lem:SgSxHasFIP}
Let $M_n$ be a sequence of finite abelian groups with exponents dividing $m$. Let $g_{1,n},\dots,g_{k,n}$ be automorphisms of $M_n$ for each $n$. Then $g_i=\lim_{n\to\omega}g_{i,n}$ is an automorphism of $\prod_{n\to\omega}M_n$ for each $i$.

For any $x_1,\dots,x_k\in\prod_{n\to\omega}M_n$ and any $g_1,\dots,g_k$-invariant homomorphism $\phi:\prod_{n\to\omega}M_n\to\Z/m\Z$, we can find a $g_1,\dots,g_k$-invariant homomorphism $\psi\in\prod_{n\to\omega}M_n^*$ such that $\psi(x_i)=\phi(x_i)$ for all $i$.
\end{lem}
\begin{proof}
Suppose $x_i=\lim_{n\to\omega}x_{i,n}$.

For each $n$, let $N_n$ be the subgroup generated by $[M_n,g_{1,n}],\dots,[M_n,g_{k,n}]$.  For simplicity of notation, let $N=\prod_{n\to\omega}N_n$. Let $q_n:M_n\to M_n/N_n$ be the quotient map.

Then if $\phi$ is $g_1,\dots,g_k$-invariant, then $\prod_{n\to\omega}N_n\subseteq\Ker(\phi)$. So we have an induced homomorphism $\phi':\prod_{n\to\omega}M_n/N_n\to\Z/m\Z$. By Lemma~\ref{lem:SxHasFIP}, we can find an element $\psi'=\lim_{n\to\omega}\psi'_n\in\prod_{n\to\omega}(M_n/N_n)^*$ such that for all $i$,
\[\phi'(x_i\bmod N)=\psi'(x_i\bmod N)=\psi'(\lim_{n\to\omega}q_n(x_{i,n}))=\lim_{n\to\omega}(\psi'_n\circ q_n)(x_{i,n}).
\] 
Let $\psi=\lim_{n\to\omega}\psi'_n\circ q_n$, then $\psi(x_i)=\phi'(x_i\bmod N)=\phi(x_i)$ as desired. Furthermore, each homomorphism $\psi'_n\circ q_n$ is invariant under $g_{1,n},\dots,g_{k,n}$. Thus $\psi$ is invariant under $g_{1},\dots,g_{k}$.
\end{proof}

Now we want to apply the restricted Burnside coproduct on the homomorphisms $\phi\in\prod_{n\to\omega}M_n^*$ and $\phi_\U$ for ultrafilters $\U$ on $\prod_{n\to\omega}M_n^*$. Here are the constructions.

For each element $\lim_{n\to\omega}\phi_n\in\prod_{n\to\omega}M_n^*$, note that each $\phi_n:M_n\to\Z/m\Z$ induces a homomorphism $\mathcal{B}_d(\phi_n):\mathcal{B}_d(M_n)\to\mathcal{B}_d(\Z/m\Z)=B(d,m)$. As a result, we make the following definition.

\begin{defn}
Given $\phi=\lim_{n\to\omega}\phi_n\in\prod_{n\to\omega}M_n^*$, we define its restricted Burnside power to be the homomorphism \[\mathcal{B}_d(\phi):=\lim_{n\to\omega}\mathcal{B}_d(\phi_n).\]
This is a homomorphism from $\prod_{n\to\omega}\mathcal{B}_d(M_n)$ to $B(d,m)$.
\end{defn}

\begin{defn}
For any ultrafilter $\mathcal{U}$ on the set $\prod_{n\to\omega}M_n^*$, we may define 
\[\mathcal{B}_d(\phi_\U):=\lim_{\phi\to\U}\mathcal{B}_d(\phi).\]
such that for all $x\in \prod_{n\to\omega}\mathcal{B}_d(M_n)$, we have
\[\mathcal{B}_d(\phi_\U)(x)=\lim_{\phi\to\U}\mathcal{B}_d(\phi)(x).\]
\end{defn}

The image $\lim_{\phi\to\U}\mathcal{B}_d(\phi)(x)$ is well-defined because $B(d,m)$ is a finite group.

Then under these definitions, the properties of homomorphisms from $\prod_{n\to\omega}M_n$ to $\Z/m\Z$ will carry over to homomorphisms from $\prod_{n\to\omega}\mathcal{B}_d(M_n)$ to $\mathcal{B}_d(\Z/m\Z)=B(d,m)$. Two particularly notable properties are surjectivity and invariance under a group action.

\begin{prop}
\label{prop:USurjPhiSurj}
Fix a sequence of finite groups $M_n$ of exponents dividing $m$. For any ultrafilter $\mathcal{U}$ on the set $\prod_{n\to\omega}M_n^*$, if $\phi_\U$ is surjective onto $\Z/m\Z$, then $\mathcal{B}_d(\phi_\U)$ is also surjective onto $B(d,m)$.
\end{prop}
\begin{proof}
Let $\iota:\Z/m\Z\to B(d,m)$ be the $i$-th coprojection map. Let $\iota_n:M_n\to \mathcal{B}_d(M_n)$ be the corresponding $i$-th coprojection map. We claim that $\mathcal{B}_d(\phi_\U)\circ(\lim_{n\to\omega}\iota_n)=\iota\circ \phi_\U$. 

If the claim is true, then we can pick any $x\in \prod_{n\to\omega}M_n$ that is a pre-image of $1\in\Z/m\Z$ under the map $\phi_\U$. Then $\mathcal{B}_d(\phi_\U)\circ(\lim_{n\to\omega}\iota_n)(x)=\iota\circ \phi_\U(x)=\iota(1)$ is the $i$-th generator of $B(d,m)$. So each generator of $B(d,m)$ is in the image of $\mathcal{B}_d(\phi_\U)$, hence $\mathcal{B}_d(\phi_\U)$ is surjective.

We now prove our claim. For any homomorphism $\phi_n:M_n\to\Z/m\Z$, by definition of the restricted Burnside coproduct functor, we have $\iota\circ\phi_n=\mathcal{B}_d(\phi_n)\circ\iota_n$.

\begin{figure}[h!]
\centering
\begin{tikzcd}
M_n\arrow{r}{\phi_n}\arrow{d}{\iota_n} & \Z/m\Z\arrow{d}{\iota}\\
\mathcal{B}_d(M_n)\arrow{r}{\mathcal{B}_d(\phi_n)} & B(d,m)
\end{tikzcd}
\end{figure}

Taking ultralimit $n\to\omega$, we have $\iota\circ\phi=\mathcal{B}_d(\phi)\circ(\lim_{n\to\omega}\iota_n)$ for any $\phi=\lim_{n\to\omega}\phi_n\in\prod_{n\to\omega}M_n^*$.

\begin{figure}[h!]
\centering
\begin{tikzcd}
\prod_{n\to\omega}M_n\arrow{r}{\phi}\arrow{d}{\lim_{n\to\omega}\iota_n} & \Z/m\Z\arrow{d}{\iota}\\
\prod_{n\to\omega}\mathcal{B}_d(M_n)\arrow{r}{\mathcal{B}_d(\phi)} & B(d,m)
\end{tikzcd}
\end{figure}

Note that the right side of the diagram stays unchanged because the groups on the right side are all finite.

So we have the calculation
\[
\mathcal{B}_d(\phi_\U)\circ(\lim_{n\to\omega}\iota_n)=
\lim_{\phi\to\U}[\mathcal{B}_d(\phi)\circ(\lim_{n\to\omega}\iota_n)]=
\lim_{\phi\to\U}\iota\circ\phi=
\iota\circ(\lim_{\phi\to\U}\phi)
=\iota\circ\phi_\U.
\]

Here the ultralimit operator $\lim_{\phi\to\U}$ and the homomorphism $\iota$ commute because $\iota$ is a homomorphism between finite groups independent of $\phi$.
\end{proof}
%\begin{proof}
%Let $M_{n,1},\dots,M_{n,d}$ be the images of the canonical embedding of $M_n$ in $\mathcal{B}_d(M_n)$. Let $e_1,\dots,e_d$ be the standard generators of $B(d,m)$.
%
%Suppose $\phi_\U:\prod_{n\to\omega}M_n\to\Z/m\Z$ is surjective. Let $x=\lim_{n\to\omega}x_n\in \prod_{n\to\omega}M_n$ be any element that is mapped to the generator $1\in\Z/m\Z$. Let $x_{n,i}$ be the image of $x_n$ in $M_{n,i}$ via the $i$-th canonical embedding of $M_n$ in $\mathcal{B}_d(M_n)$. 
%
%Let $S_x$ be the set of all $\phi\in\prod_{n\to\omega}M_n^*$ such that $\phi(x)=1$. Then $S_x\in\U$. Now for any $\phi\in S_x$, say $\phi=\lim_{n\to\omega}\phi_n$, then $\lim_{n\to\omega}\phi_n(x_n)=1$. Then $\phi_n(x_n)=1$ implies that for all $i=1,...,d$, we have $\mathcal{B}_d(\phi_n)(x_{n,i})=e_i$. Hence $\mathcal{B}_d(\phi)(\lim_{n\to\omega}x_{n,i})=\lim_{n\to\omega}\mathcal{B}_d(\phi_n)(x_{n,i})=e_i$. Since this is true for all $\phi\in S_x$ and $S_x\in\U$, we see that $\mathcal{B}_d(\phi)_\U(\lim_{n\to\omega}x_{n,i})=\lim_{\phi\to\U}\mathcal{B}_d(\phi)(\lim_{n\to\omega}x_{n,i})=e_i$. So $e_i$ is in the image of $\mathcal{B}_d(\phi_\U)$.
%
%Since $e_i$ is in the image of $\mathcal{B}_d(\phi_\U)$ for all $i$, and they generate $B(d,m)$, the homomorphism $\mathcal{B}_d(\phi_\U)$ is surjective.
%\end{proof}

\begin{prop}
\label{prop:UGinvPhiInv}
Fix a sequence of finite groups $M_n$ of exponents dividing $m$. Suppose we have a finite group $G_n$ acting on $M_n$ for each $n$. So we have an action of $G=\prod_{n\to\omega}G_n$ on $\prod_{n\to\omega}M_n$.

For any $G$-invariant ultrafilter $\mathcal{U}$ on the set $\prod_{n\to\omega}M_n^*$, then both $\phi_\U$ and $\mathcal{B}_d(\phi)_\U$ are $G$-invariant.
\end{prop}
\begin{proof}
Suppose $\U$ is $G$-invariant. For any $g\in G$, let $S_g$ be the subset of $g$-invariant elements of $\prod_{n\to\omega}M_n^*$. Then $S_g\in\U$ for all $G$. So for each $x\in\prod_{n\to\omega}M_n$, we have $\phi_\U(x^g)=\lim_{\phi\to\U}\phi(x^g)=\lim_{\phi\to\U}\phi(x)=\phi_\U(x)$.

Now for each $\phi\in S_g$, say $\phi=\lim_{n\to\omega}\phi_n$ and $g=\lim_{n\to\omega}g_n$. Then $\phi_n$ is $g_n$-invariant for $\omega$-almost all $n$. But by Proposition~\ref{prop:BdActionKernel}, this means $\mathcal{B}_d(\phi_n)$ is $g_n$-invariant for $\omega$-almost all $n$. So $\mathcal{B}_d(\phi)=\lim_{n\to\omega}\mathcal{B}_d(\phi_n)$ is $g$-invariant.

So $\mathcal{B}_d(\phi)$ is $g$-invariant for all $\phi\in S_g$, i.e., for $\U$-almost all $\phi$. So for each $x\in\prod_{n\to\omega}\mathcal{B}_d(M_n)$, we have $\mathcal{B}_d(\phi)_\U(x^g)=\lim_{\phi\to\U}\mathcal{B}_d(\phi)(x^g)=\lim_{\phi\to\U}\mathcal{B}_d(\phi)(x)=\mathcal{B}_d(\phi)_\U(x)$.

Since $g$ is an arbitrary element of $G$, both $\phi_\U$ and $\mathcal{B}_d(\phi)_\U$ are $G$-invariant.
\end{proof}

\subsection{Reduction to the cyclic case}
\label{subsec:Cyclic}

Recall that our goal is to realize the restricted Burnside group $B(d,m)$ as a quotient of an ultraproduct of finite perfect groups. In this subsection, we shall use the fact $B(d,m)=\mathcal{B}_d(\Z/m\Z)$ to reduce our goal to this: it is enough to realize the cyclic group $\Z/m\Z$ as a quotient of an ultraproduct of ``nice'' finite perfect groups.

In particular, we shall show that it is enough to have the following lemma. We delay its proof to later subsections.

\begin{lem}
\label{lem:CyclicLem}
Pick any non-principal ultrafilter $\omega$ on $\N$ and any positive integer $m$. Then we can find a sequence of finite groups $\{G_n\}_{n\in\N}$, each $G_n$ acting on the corresponding finite abelian group $M_n$ whose exponent divides $m$, such that they satisfy the following conditions.
\begin{enumerate}
\item $M_n\rtimes G_n$ is perfect for each $n$.
\item There is a $\prod_{n\to\omega}G_n$-invariant surjective homomorphism $\psi:\prod_{n\to\omega} M_n\to\Z/m\Z$.
\end{enumerate}
\end{lem}

Note that the conditions here guarantee the following corollary.

\begin{cor}
Any cyclic group $\Z/m\Z$ can be realized as an abstract quotient of an ultraproduct of finite perfect groups.
\end{cor}
\begin{proof}
Pick any non-principal ultrafilter $\omega$ on $\N$ and any positive integer $m$. Let $\{M_n\}_{n\in\N}$ and $\{G_n\}_{n\in\N}$ and $\psi$ be the sequences of groups and the homomorphism in Proposition~\ref{lem:CyclicLem}.

Since $\psi$ is $\prod_{n\to\omega}G_n$-invariant, by Lemma~\ref{lem:GrpActExt} below, it has an extension $\td\psi:(\prod_{n\to\omega}M_n)\rtimes(\prod_{n\to\omega}G_n)\to\Z/m\Z$. So $\td\psi$ is surjective, and as ultraproduct preserves semi-direct product, the domain of $\td\psi$ is canonically isomorphic to $\prod_{n\to\omega}(M_n\rtimes G_n)$, an ultraproduct of finite perfect groups.
\end{proof}

\begin{lem}
\label{lem:GrpActExt}
Suppose we have a group $G$ acting on another group $M$, and a $G$-invariant homomorphism $\psi:M\to B$ to some group $B$. Then there is an extension $\td\psi:M\rtimes G\to B$ of $\psi$.
\end{lem}
\begin{proof}
Let $N=[M,G]$. Then since $\psi$ is $G$-invariant, $N\subseteq\Ker(\psi)$. So the map $\psi$ factors into a quotient map $q:M\to M/N$ and a map $\psi':M/N\to B$.

For any $x,y\in M$ and $g,h\in G$, we have $[x,g]^h=[x^h,g^h]\in N$ and $[x,g]^y=[xy,g][y,g]^{-1}\in N$, so $N$ is a $G$-invariant normal subgroup of $M\rtimes G$. So $N\rtimes G$ is a subgroup of $M\rtimes G$. Furthermore, for any $x\in N,g\in G,y\in M,h\in G$, then in the group $M\rtimes G$ we have the following calculation:
\[
(xg)^y=x^yg^y=x^yg^y g^{-1}g=x^y[y,g^{-1}]g\in N\rtimes G.
\]
\[
(xg)^h=x^hg^h=xx^{-1}x^hg^h=x[x,h]g^h\in N\rtimes G.
\]
Hence we see that $N\rtimes G$ is a normal subgroup of $M\rtimes G$. But for each $x\in M$ and $g\in G$, then $[x,g]=(x^{-1}g^{-1}x)g$ is in the normal closure of $G$. So $N\rtimes G$ is in the normal closure of $G$, and hence it is the normal closure of $G$. 

Let $\td q$ be the quotient homomorphism from $M\rtimes G$ to $(M\rtimes G)/(N\rtimes G)\simeq M/N$. Then we have a commutative diagram:
\[ \begin{tikzcd}[column sep=large]
N\arrow[hook]{r} \arrow[hook]{d} &  M\arrow[hook]{d}\arrow[two heads]{r}{q} & M/N \arrow{d}{\simeq} \arrow{r}{\psi'}& B\\
N\rtimes G\arrow[hook]{r} & M\rtimes G\arrow[two heads]{r}{\td q} & (M\rtimes G)/(N\rtimes G) &
\end{tikzcd}\]

So we have a map $\td\psi:M\rtimes G\to B$ extending $\psi:M\to B$.
\end{proof}

Assuming Lemma~\ref{lem:CyclicLem}, which builds a homomorphism onto $\Z/m\Z$, now we show that it induces a similar construction which shall induce a homomorphism onto $B(d,m)=\mathcal{B}_d(\Z/m\Z)$.

\begin{lem}
\label{lem:AnyLem}
Pick any non-principal ultrafilter $\omega$ on $\N$ and any positive integers $d,m$. Let $\{M_n\}_{n\in\N}$ and $\{G_n\}_{n\in\N}$ and $\psi$ be the sequences of groups and the homomorphism in Lemma~\ref{lem:CyclicLem}. Then they satisfy the following conditions:
\begin{enumerate}
\item $\mathcal{B}_d(M_n)\rtimes G_n$ is perfect for each $n$.
\item We can find a $\prod_{n\to\omega}G_n$-invariant surjective homomorphism from $\prod_{n\to\omega} \mathcal{B}_d(M_n)$ to $B(d,m)$.
\end{enumerate}
\end{lem}
\begin{proof}
By Proposition~\ref{prop:BdActionPerfect}, $\mathcal{B}_d(M_n)\rtimes G_n$ is perfect for each $n$. 

Let $G=\prod_{n\to\omega}G_n$. Since $\psi$ is $G$-invariant, by Proposition~\ref{prop:AnyToU}, we can find a $G$-invariant ultrafilter $\U$ on $\prod_{n\to\omega}M_n^*$ such that $\psi=\phi_\U$. Then by Proposition~\ref{prop:UGinvPhiInv}, $\mathcal{B}_d(\phi)_\U$ is also $G$-invariant. Furthermore, by Proposition~\ref{prop:USurjPhiSurj}, since $\phi_\U=\psi$ is surjective, $\mathcal{B}_d(\phi)_\U$ is also surjective. So we are done.
\end{proof}

\begin{cor}
For any positive integer $d,m$, the restricted Burnside group $B(d,m)$ can be realized as an abstract quotient of an ultraproduct of finite perfect groups.
\end{cor}
\begin{proof}
Pick any non-principal ultrafilter $\omega$ on $\N$ and any positive integer $m$. Let $\{M_n\}_{n\in\N}$ and $\{G_n\}_{n\in\N}$ and $\psi$ be the sequences of groups and the homomorphism in Lemma~\ref{lem:CyclicLem}.

Let $G=\prod_{n\to\omega}G_n$. For any positive integer $d$, by Lemma~\ref{lem:AnyLem}, we can then find the $G$-invariant surjective homomorphism $\mathcal{B}_d(\phi)_\U:\prod_{n\to\omega}\mathcal{B}_d(M_n)\to B(d,m)$. By Lemma~\ref{lem:GrpActExt}, $\mathcal{B}_d(\phi)_\U$ has an extension $\td\psi:\prod_{n\to\omega}\mathcal{B}_d(M_n)\rtimes G_n\to B(d,m)$.

Note that $\mathcal{B}_d(M_n)\rtimes G_n$ is perfect for each $n$, so $B(d,m)$ is indeed an abstract quotient of an ultraproduct of finite perfect groups.
\end{proof}

Hence Theorem~\ref{thm:mainthm2} is proven.

\subsection{Constructing the group $G_n$}
\label{subsec:Gn}

We now move on to the proof of Lemma~\ref{lem:CyclicLem}, by an explicit construction of the groups $G_n$ and $M_n$ and the homomorphism $\psi$.

Pick any two distinct prime numbers $p,q$ coprime to $m$, and $p\geq 5$. For each $n\in\N$, the alternating group $\A_p$ acts on the $\F_q$-vector space $(\F_q^n)^{p}$ by permuting the $p$ factor spaces $\F_q^n$. Let $V_n$ be the subspace of $(\F_q^n)^{p}$ of tuples $(v_1,\dots,v_p)$ with $v_1,\dots,v_p\in\F_q^n$ such that $\sum v_i=0$. Note that $V_n$ is $\A_p$-invariant.

\begin{prop}
$V_n\rtimes\A_p$ is a perfect group with commutator width at most 2.
\end{prop}
\begin{proof}
For any $\sigma\in\A_p$, we have $(v_1,\dots,v_p)^\sigma=(v_{\sigma(1)},\dots,v_{\sigma(p)})$ for any $v_1,\dots,v_p\in\F_q^n$. Now if $v_1+\dots+v_p=0$, then $v_{\sigma(1)}+\dots+v_{\sigma(p)}=0$ as well. Hence $V_n$ is an $\A_p$-invariant subspace. 

Since $p\geq 5$ and it is a prime, it must be odd. So let $\sigma$ be the $p$-cycle in $\A_p$ such that $(v_1,\dots,v_p)^\sigma=(v_p,v_1,v_2,\dots,v_{p-1})$ for any $v_1,\dots,v_p\in\F_q^n$. Note that the action by $\sigma$ induces an automorphism of $V_n$, and if $(v_1,\dots,v_p)^\sigma=(v_1,\dots,v_p)$ in $V_n$, then we must have $v_i=v_j$ for all $i,j$. Since $p,q$ are coprime, $0=v_1+\dots+v_p=pv_i$ implies that $v_i=0$ for all $i$.  So $\sigma$ induces an automorphism with only trivial fixed point. 

Now the map $v\mapsto [v,\sigma]$ is a linear automorphism of $V_n$, and its kernel is exactly the set of fixed points of $\sigma$, which must be trivial. So this map is bijective. So for any $v\in V_n$, we can find $w\in V_n$ such that $v=[w,\sigma]$. So all elements of $V_n$ are commutators of $V_n\rtimes\A_p$.

Now for the group $V_n\rtimes\A_p$, all elements in $\A_p$ are commutators since $p\geq 5$, and all elements in $V_n$ are commutators. Therefore $V_n\rtimes\A_p$ is perfect with commutator width at most 2.
\end{proof}

\begin{rem}
The bound here is tight. Suppose $p=5$. Let $\sigma=(12)(34)$ and pick any $v\in V_n$, and consider $v \sigma\in V_n\rtimes\A_p$. Suppose $v\sigma=[v_1 \sigma_1,v_2  \sigma_2]$, then $[\sigma_1,\sigma_2]=\sigma$. Then $\sigma_1(5)=\sigma_2(5)=5$ by Lemma~\ref{lem:RemarkLem} below.

Now $[v_1 \sigma_1,v_2  \sigma_2]=(L_1(L_2-I)v_1+L_1L_2(L_1^{-1}-I)v_2) \sigma$, where $L_i$ is the linear automorphism on $V_n$ induced by $\sigma_i$. But since both $L_1,L_2$ fix the fifth component, $L_1(L_2-I)$ and $L_1L_2(L_1^{-1}-I)$ must kill the fifth component. So we must have $v \sigma\neq[v_1 \sigma_1,v_2  \sigma_2]$ when $v$ has non-zero fifth component.
\end{rem}

\begin{lem}[Remark Lemma]
\label{lem:RemarkLem}
If $[\sigma_1,\sigma_2]=(12)(34)$ in $\A_5$, then $\sigma_1(5)=\sigma_2(5)=5$.
\end{lem}
\begin{proof}
The proof is basically a direct enumeration of all possible cycle-types of $\sigma_1$.

Suppose $\sigma_1$ is a $5$-cycle, then this implies that a product of two 5-cycles $\sigma=\sigma_1^{-1}$ and $\sigma'=\sigma_2^{-1}\sigma_1\sigma_2$ is $(12)(34)$.

Suppose $\sigma,\sigma'$ are two 5-cycles such that $\sigma\sigma'=(12)(34)$. Then $\sigma'(\sigma(5))=5$. By symmetry, WLOG, say $\sigma(5)=1$. Since $\sigma'(\sigma(1))=2$ and $\sigma'$ has no fixed point, we see that $\sigma(1)\neq 2$. We also know that $\sigma(1)\neq 1$ since $\sigma$ has no fixed point and $\sigma(1)\neq 5$ since $\sigma$ is a 5-cycle. Hence $\sigma(1)=3$ or $4$. By symmetry, WLOG say $\sigma(1)=3$.

Now $\sigma'(\sigma(3))=4$, so $\sigma(3)\neq 4$. But we also have $\sigma(3)\neq 3$ or $5$ since $\sigma$ is a 5-cycle, and $\sigma(3)\neq 1$ since $\sigma(5)=1$. This means $\sigma(3)=2$.

Now $\sigma(2)\neq 2,3,5$ since they are already taken by other $\sigma$-values, and $\sigma'(\sigma(2))=1$ implies $\sigma(2)\neq 1$. So $\sigma(2)=4$. And finally, this means $\sigma(4)=5$. So $\sigma=(13245)$. It is then easy to see that $\sigma'=(14235)^{-1}$. But now we see that $(13245)$ and $(14235)$ are NOT conjugate in $\A_5$. However, by definition $\sigma=\sigma_1^{-1}$ and $\sigma'=\sigma_2^{-1}\sigma_1\sigma_2$ are conjugate in $\A_5$, a contradiction. So $\sigma_1$ cannot be a 5-cycle. Similarly, $\sigma_2$ also cannot be a 5-cycle.

Now suppose $\sigma_1=\sigma$ is a 3-cycle. Then $(12)(34)$ is the product of two 3-cycles. Note that the union of the supports of the two 3-cycles must obviously include $\{1,2,3,4\}$. However, if two 3-cycles have only one element in common in their support, then their product is a 5-cycle.  So the two 3-cycles must have two elements in common. So they must both have support inside $\{1,2,3,4\}$. WLOG suppose the other fixed point of $\sigma$ is $4$. Then $\sigma(3)=1$ or $2$. WLOG suppose $\sigma(3)=1$. Then $\sigma=(123)$ and $\sigma'=(143)=(134)^{-1}$. So $\sigma_2=(234),(142)$ or $(13)(24)$. We see that in this case, indeed we have $\sigma_1(5)=\sigma_2(5)=5$.

Finally, suppose neither $\sigma_1$ nor $\sigma_2$ is a 3-cycle. Then both $\sigma=\sigma_1^{-1}$ and $\sigma'=\sigma_2^{-1}\sigma_1\sigma_2$ must be of cycle type $(2,2,1)$. Then $\sigma\sigma'=(12)(34)=[(12)(34)]^{-1}=(\sigma')^{-1}\sigma^{-1}=\sigma'\sigma$. Suppose $\sigma(5)\neq 5$, say WLOG $\sigma(5)=1$. Then since $\sigma\sigma'=(12)(34)$, we see that $\sigma'(5)=1$ as well. So $\sigma=(15)(ab)$ and $\sigma'=(15)(cd)$, where $a,b,c,d\in\{2,3,4\}$. But then $\sigma\sigma'$ is a 3-cycle, a contradiction. So $\sigma(5)=5$, and therefore $\sigma_1(5)=\sigma_2(5)=5$ as desired.
\end{proof}

We shall set $G_n=V_n\rtimes\A_p$. Let $V=\prod_{n\to\omega} V_n$ and $G=\prod_{n\to\omega} G_n$.

\begin{prop}
$G=V\rtimes\A_p$.
\end{prop}
\begin{proof}
Note that the ultraproduct of semi-direct products is the semidirect product of ultraproducts. And since $\A_p$ is finite, we have a canonical identification $\prod_{n\to\omega}\A_p=\A_p$.
\end{proof}

We now establish some properties of $G$ to be used later. Note that $\F_q^n$ has a natural ``inner product''. If $v,w\in \F_q^n$, let them be $v=(v_1,\dots,v_{n}),w=(w_1,\dots,w_{n})$ for $v_1,\dots,v_{n},w_1,\dots,w_{n}\in\F_q$, then we define $\<v,w\>=\sum v_iw_i$. Note that this defines a non-degenerate symmetric bilinear form on $\F_q^n$. This then induces a symmetric bilinear form on $(\F_q^n)^p$ such that $\<(v_1,\dots,v_p),(w_1,\dots,w_p)\>=\sum\<v_i,w_i\>$ for $v_1,\dots,v_p,w_1,\dots,w_p\in\F_q^n$.
%%
%%This is non-degenerate on $\F_q^n$ because for any $v\in \F_q^n$, let $e_1,\dots,e_n\in\F_q^n$ be the standard basis, then $\<v,e_i\>=0$ for all $i$ implies that $v=0$. 

\begin{prop}
$\<v^{(\sigma^{-1})},w\>=\<v,w^\sigma\>$ for all $v,w\in (\F_q^n)^p$ and $\sigma\in\A_p$.
\end{prop}
\begin{proof}
$\<(v_1,\dots,v_p)^{(\sigma^{-1})},(w_1,\dots,w_p)\>=\sum v_{\sigma^{-1}(i)}w_i=\sum v_iw_{\sigma(i)}=\<(v_1,\dots,v_p),(w_1,\dots,w_p)^\sigma\>$.
\end{proof}

\begin{prop}
The induced bilinear form on $(\F_q^n)^p$ restricts to a symmetric non-degenerate bilinear form on $V_n$.
\end{prop}
\begin{proof}
We only needs to verify that it is non-degenerate. Suppose for some $v\in V_n$, $\<v,w\>=0$ for all $w\in V_n$. Then if $v=(v_1,\dots,v_p)$ for $v_1,\dots,v_p\in \F_q^n$, pick any $w'\in\F_q^n$, then $\<v,(w',-w',0,\dots,0)\>=0$ implies that $\<v_1,w'\>=\<v_2,w'\>$ for all $w'\in\F_q^n$. Since $\<-,-\>$ is non-degenerate on $\F_q^n$, we have $v_1=v_2$. Similarly we have $v_i=v_j$ for all $i,j$. But since $v\in V_n$, we have $\sum v_i=0$, so $pv_i=0$. Finally, since $p,q$ are coprime, we have $v_i=0$ for all $i$. So $v=0$.
\end{proof}

\begin{lem}
The symmetric non-degenerate bilinear forms on $V_n$ induce a symmetric non-degenerate bilinear form on the ultraproduct $V=\prod_{n\to\omega} V_n$.
\end{lem}
\begin{proof}
Taking the ultralimit of $\<-,-\>:V_n\times V_n\to\F_q$, we obtain a symmetric bilinear map $\<-,-\>:V\times V\to\F_q$. Let us verify that it is indeed non-degenerate.

Suppose $\lim_{n\to\omega}v_n$ is non-zero in $V$. Then $v_n\neq 0$ for $\omega$-almost all $n$. For each $v_n\neq 0$, since $\<-,-\>$ is non-degenerate, we can find $w_n\in V_n$ such that $\<v_n,w_n\>\neq 0$. If $v_n=0$, we pick $w_n$ arbitrarily. So $\<v_n,w_n\>\neq 0$ for $\omega$-almost all $n$ by construction.

Then since our bilinear form takes value in a finite set $\F_q$, we have $\<\lim_{n\to\omega}v_n,\lim_{n\to\omega}v_n\>=\lim_{n\to\omega}\<v_n,w_n\>\neq 0$. So $\<-,-\>$ is indeed non-degenerate.
\end{proof}

We now study subspaces of $V$ with finite codimension, which would be useful later. We first make a definition of a ``closed'' subspace.

\begin{defn}
Given a vector space $V$ with a symmetric non-degenerate bilinear form, we say $v,w\in V$ are \emph{perpendicular} if $\<v,w\>=0$.

We define the \emph{orthogonal complement} of a subset $S\subseteq V$ to be the set $S^\perp:=\{v\in V:\<v,w\>=0\text{ for all $w\in S$}\}$, i.e., the subspace of elements perpendicular to all elements of $S$.

We say a subspace $W$ of $V$ is \emph{closed} if $W=S^\perp$ for some subset $S$. 
\end{defn}

\begin{lem}
If $S\subseteq T$ for two subsets $S,T$ in a vector space $V$ with a symmetric non-degenerate bilinear form, then $T^\perp\subseteq S^\perp$ and $S\subseteq S^{\perp\perp}$.
\end{lem}
\begin{proof}
If $v\in T^\perp$, then $v$ is perpendicular to all elements of $T$, which includes all elements of $S$. Hence $v\in S^\perp$.

If $v\in S$, then it is by definition perpendicular to all elements of $S^\perp$, so $s\in S^{\perp\perp}$.
\end{proof}

\begin{lem}
\label{lem:ClosedSubspace}
Given a vector space $V$ with a symmetric non-degenerate bilinear form, then we have the following results.
\begin{enumerate}
\item A subspace $W$ is closed if and only if $W=W^{\perp\perp}$.
\item If $W$ is a closed subspace with finite codimension $c$, then $W^\perp$ has dimension $c$.
\item Finite dimensional subspaces are closed.
\end{enumerate}
\end{lem}
\begin{proof}
Suppose $W=S^\perp$, then $S\subseteq S^{\perp\perp}=W^\perp$, and hence $W^{\perp\perp}\subseteq S^\perp=W$. On the other hand, $W\subseteq W^{\perp\perp}$. So we see $W=W^{\perp\perp}$.

Now suppose $W$ is a subspace with finite codimension $c$. Pick linearly independent vectors $v_1,\dots,v_c\in V$ such that $W,v_1,\dots,v_c$ span $V$. Then since the bilinear form is non-degenerate, $W^\perp\cap\{v_1\}^\perp\cap\dots\cap\{v_c\}^\perp$ must be the trivial. 

Since each $v_i$ is non-zero and the bilinear form is non-degenerate, the map $w\mapsto\<v,w\>$ is a surjective linear map from $V$ to $\F_q$ with kernel $\{v_i\}^\perp$, so $\{v_i\}^\perp$ is a closed subspace of codimension $1$.

Since $W,v_1,\dots,v_c$ are linearly independent, the subspaces $W^\perp,\{v_1\}^\perp,\dots,\{v_c\}^\perp$ intersect each other transversally. 
%So as $W^\perp$ intersects with $\{v_1\}^\perp,\dots,\{v_c\}^\perp$ one by one, each time its dimension would decrease by exactly 1. 
Since $W^\perp\cap\{v_1\}^\perp\cap\dots\cap\{v_c\}^\perp$ has dimension zero, $W^\perp$ has dimension $c$.

Finally, suppose $W$ is finite dimensional. Say it has a basis $v_1,\dots, v_k$. Then $W^\perp=\{v_1\}^\perp\cap\dots\cap\{v_k\}^\perp$, so it is a closed subspace of codimension $k$. So $\dim W^{\perp\perp}=k$. Since we also have $\dim W=k$ and $W\subseteq W^{\perp\perp}$, we conclude that $W=W^{\perp\perp}$, i.e., it is a closed subspace.
\end{proof}

We now go back to our discussion of the group $G=V\rtimes\A_p$ where $V=\prod_{n\to\omega} V_n$.

\begin{prop}
\label{prop:ClosedSubspaceAreUltra}
If $W$ is a closed $\A_p$-invariant subspace of $V$ of codimension $c$, then for $\omega$-almost all $n$, we can find closed $\A_p$-invariant subspace $W_n$ of $V_n$ of codimension $c$ such that $W=\prod_{n\to\omega} W_n$.
\end{prop}
\begin{proof}
By Lemma~\ref{lem:ClosedSubspace}, $W^\perp$ is $c$-dimensional. Let $v_1,\dots, v_c$ be a basis for $W^\perp$, say $v_i=\lim_{n\to\omega}v_{i,n}\in V$. Let $W_n$ be the orthogonal complement of $v_{1,n},\dots,v_{c,n}$. Since $v_1,\dots, v_c$ are linearly independent, and there are only finitely many possible linear combinations among $c$ vectors, we see that $v_{1,n},\dots,v_{c,n}$ are linearly independent for $\omega$-almost all $n$. So $W_n$ has codimension $c$.

Let us show that $\prod_{n\to\omega} W_n$ is exactly the orthogonal complement of $\{v_1,\dots,v_c\}$.

Fix $1\leq i\leq c$. For each $w=\lim_{n\to\omega}w_n\in\prod_{n\to\omega} W_n$, then since $w_n\in W_n$ for all $n$, we see that $w_n\perp v_{i,n}$. So $\<w,v_i\>=\lim_{n\to\omega}\<w_n,v_{i,n}\>=0$.

Conversely, suppose $w=\lim_{n\to\omega}w_n\in V$ is orthogonal to all $v_i$. Then $\lim_{n\to\omega}\<w_n,v_{i,n}\>=0$ for all $i$. Since there are only finitely many possible $i$, therefore for $\omega$-almost all $n$, $\<w_n,v_{i,n}\>=0$ for all $i$, and hence $w_n\in W_n$. So $w\in \prod_{n\to\omega} W_n$.

In conclusion, we have seen that $W=\prod_{n\to\omega} W_n$.

Now we need to show that $W_n$ is $\A_p$-invariant for $\omega$-almost all $n$. Suppose for contradiction that this is not the case. Then for $\omega$-almost all $n$, $W_n$ is NOT $\A_p$-invariant, and we can find $w_n\in W_n$ such that $w_n^{\sigma_n}\notin W_n$ for some $\sigma_n$. So $\lim_{n\to\omega}w_n\in W$ but $\lim_{n\to\omega}w_n^{\sigma_n}\notin W$.

However, since $\A_p$ is a finite set, let $\sigma=\lim_{n\to\omega}\sigma_n$. Then $\lim_{n\to\omega}w_n^{\sigma_n}=\lim_{n\to\omega}w_n^{\sigma}=(\lim_{n\to\omega}w_n)^\sigma\in W$, a contradiction. So for $\omega$-almost all $n$, $W_n$ must be an $\A_p$-invariant subspace.
\end{proof}

\begin{prop}
\label{prop:CodimWwithSmallerInvCodim}
Suppose $W$ is a closed subspace of $V$ with finite codimension. Then we can find a closed $\A^p$-invariant subspace $W'\subseteq W$ with finite codimension in $V$.
\end{prop}
\begin{proof}
Let $v_1,\dots,v_k$ be a basis for $W^\perp$. Then the orbit $\O(v_i)$ for each $v_i$ under the $\A^p$ action has at most $\frac{p!}{2}$ elements. So the union of all these orbits has at most $\frac{k(p!)}{2}$ elements, and they span a finite dimensional $\A_p$-invariant subspace $U$ containing $W^\perp$. So $U^\perp$ is an $\A_p$-invariant closed subspace.

Note that $U$ is finite dimensional, so by Lemma~\ref{lem:ClosedSubspace}, it must be a closed subspace. So $U^\perp$ has codimension $\dim U$, which is at most $\frac{k(p!)}{2}$.
\end{proof}

Finally, here is a lemma we need for future use.

\begin{lem}
\label{lem:WnOrbitp}
If $W_n\subseteq V_n$ is a subspace of codimension $c$, and $n>c(p-1)$. Then we can find $w\in W_n$ such that its orbit $\O(w)$ in $V_n$ under the $\A_p$ action has exactly $p$ elements.
\end{lem}
\begin{proof}
Since $V_n$ is finite dimensional, all subspaces are closed. Let $W_n$ be the orthogonal complement of $v_1,\dots,v_c\in V_n$. Let $v_i=(v_{i,1},\dots,v_{i,p})\in V_n\subseteq (\F_q^n)^p$, where $v_{i,j}\in\F_q^n$ and $\sum_j v_{i,j}=0$ for each $i$.

Now consider the span of these $v_{i,j}$ in $\F_q^n$. Since $\sum_j v_{i,j}=0$ for each $i$, this subspace has dimension at most $c(p-1)$. Let $U$ be the orthogonal complment of this subspace in $\F_q^n$. Since $n>c(p-1)$, $U$ is non-trivial. Pick any non-zero $u\in U$. Consider $w=(u,\dots,u,(1-p)u)\in V_n$. Since $u\perp v_{i,j}$ for all $i,j$, we see that $w\perp v_i$ for all $i$. Hence $w\in W_n$. 

Now $\A_p$ acts on $V_n$ by permuting the $p$ components. Since $p,q$ are coprime, we must have $(1-p)u=u-pu\neq u$. Hence $\O(w)$ has exactly $p$ elements.
\end{proof}

\subsection{Constructing the $\Z/m\Z$-modules $M_n$ and an action by $G_n$}
\label{subsec:Mn}

We now move on to the construction of $M_n$ in Lemma~\ref{lem:CyclicLem}. Note that the category of finite abelian groups whose exponent divides $m$ is the same as the category of finite $\Z/m\Z$-modules. So we construct $M_n$ as $\Z/m\Z$-modules. We use the same $p,q$, and $G_n$ as in the last section.

Consider the free $\Z/m\Z$-module $(\Z/m\Z)^q$. Let $B$ be the submodule of $(\Z/m\Z)^{q}$ of tuples $(t_1,\dots,t_q)$ with $t_1,\dots,t_q\in\Z/m\Z$ such that $\sum t_i=0$. Consider the module automorphism $f:B\to B$ such that $f:(t_1,\dots,t_q)\mapsto (t_q,t_1,\dots,t_{q-1})$.

\begin{prop}
$B$ is a free module, and $f$ only has the trivial fixed-point and has order $q$.
\end{prop}
\begin{proof}
Obviously $f^q=\id$ the identity automorphism, and if $f((t_1,\dots,t_q))=(t_1,\dots,t_q)$ for an element $(t_1,\dots,t_q)\in B$, then we see that all $t_i$ are the same. But since $\sum t_i=0$ and $q,m$ are coprime, we see that $t_i=0$ for all $i$. Hence $f$ only has the trivial fixed-point.

Let us now show that $B$ is a free module. Set $b_1=(1,-1,0,\dots, 0)$, and set $b_{i+1}=f(b_i)$ for each $i\in\Z/q\Z$. We claim that $b_1,\dots,b_{q-1}$ form a basis.

For any $b=(t_1,\dots,t_q)\in B$, then since $t_q=-t_1-\dots-t_{q-1}$, we have 
\[
b=t_1b_1+(t_1+t_2)b_2+\dots+(t_1+\dots+t_{q-1})b_{q-1}.
\]
So $b_1,\dots,b_{q-1}$ are spanning. 

Now suppose $\sum t_ib_i=(0,\dots,0)$. Then since the first component of $\sum t_ib_i$ is zero, and only $b_1$ contributes to this component, we see that $t_1=0$. Moving on inductively, we see that $t_i=0$ for all $i$, and hence $b_1,\dots,b_{q-1}$ are linearly independent.
\end{proof}

\begin{cor}
$f-\id$ is also an automorphism of $B$, where $\id$ is the identity automorphism of $B$.
\end{cor}

Let $M_n=B^{V_n-\{0\}}$, the $\Z/m\Z$-module of all set-functions from $V_n-\{0\}$ to $B$. Then clearly $M_n$ is a free $\Z/m\Z$ module. 

We let $V_n$ act on $M_n$ such that $x^v(w)=f^{\<v,w\>}(x(w))$ for each $x\in M_n,v\in V_n,w\in V_n-\{0\}$. This is well defined because $\<v,w\>\in\F_q=\Z/q\Z$ and $f$ has order $q$. 

We also let $\A_p$ act on $M_n$ such that $x^\sigma(v)=x(v^{(\sigma^{-1})})$ for each $x\in M_n,v\in V_n-\{0\},\sigma\in\A_p$.

\begin{prop}
We have a group action of $G_n=V_n\ltimes \A_p$ on $M_n$ induced from the group actions of $V_n$ and $\A_p$ on $M_n$.
\end{prop}
\begin{proof}
\[
x^{\sigma^{-1}v\sigma}(w)=x^{\sigma^{-1}v}(w^{(\sigma^{-1})})=f^{\<v,w^{(\sigma^{-1})}\>}(x^{\sigma^{-1}}(w^{(\sigma^{-1})}))=f^{\<v,w^{(\sigma^{-1})}\>}(x(w))=f^{\<v^\sigma,w\>}(x(w))=x^{(v^\sigma)}(w). 
\]
So the group action is well-defined.
\end{proof}

\subsection{$M_n\rtimes G_n$ is perfect}
\label{subsec:Pn}

We not only want this to be perfect, but we also want $M_n\rtimes G_n$ to have increasing commutator width. Otherwise, $\prod_{n\to\omega}M_n\rtimes G_n$ will also have bounded commutator length and become a perfect group, in which case any $\psi:\prod_{n\to\omega}M_n\rtimes G_n\to\Z/m\Z$ would be trivial.

\begin{prop}
\label{prop:MnSemiGnPerfect}
The group $M_n\rtimes G_n$ is perfect. 
\end{prop}
\begin{proof}
We already know that $G_n=V_n\rtimes \A_p$ is perfect. So we just need to show that $M_n$ is also in the commutator subgroup of $M_n\rtimes G_n$. We are going to show that $M_n=[M_n,V_n]$.

Pick any $v\in V_n-\{0\}$ and any $b\in B$, and let $x_{v\to b}\in M_n$ be the function such that $x_{v\to b}(v)=b$ and $x_{v\to b}(v')=0$ for all $v'\neq v$. Note that such elements generate $M_n$. So we only need to show that $x_{v\to b}$ for all $v\in V_n-\{0\}$ and all $b\in B$ are in $[V_n,M_n]$.

Pick any $w\in V_n$ such that $w$ is NOT perpendicular to $v$, say $\<w,v\>=1$. Then for any $v'\neq v$, we have 
\[
[x_{v\to b},w](v')=(x_{v\to b}^{w}-x_{v\to b})(v')=f(0)-0=0.
\]
And we also have
\[
[x_{v\to b},w](v)=(x_{v\to b}^{w}-x_{v\to b})(v)=f(b)-b=(f-\id)(b).
\] 
Here $\id$ is the identity automorphism on $B$. So $[x_{v\to b},w]=x_{v\to(f-\id)(b)}$. 

Since $f-\id$ is also an automorphism, we see that for any $b'\in B$, we can find $b\in B$ such that $b'=(f-\id) b$, so 
\[
x_{v\to b'}=[x_{v\to b},w]\in [V_n,M_n].\] 
So we are done.
\end{proof}

We make the following useful definition.

\begin{defn}
For each $x\in M_n$, we define the \emph{null set} for $x$ as 
\[\nul(x):=\{0\}\cup\{v\in V_n-\{0\}\mid x(v)=0\}.\]
\end{defn}

We include zero for convenience. Under certain conditions, we want $\nul(x)$ to be some subspace of $V_n$.

\begin{lem}
\label{lem:AVMCalculation}
For any $\sigma\in\A_p$ $v\in V_n$ and $x\in M_n$, we have 
\begin{align*}
[x,v \sigma]=&[x^v,\sigma]+[x,v],\\
\{v\}^\perp\subseteq&\nul([x,v]). 
\end{align*}
And for each $\A_p$-orbit $S$ of $V_n-\{0\}$, we have 
\[
\sum_{w\in S}[x,\sigma](w)=0.
\] 
\end{lem}
\begin{proof}
\[
[x,v \sigma]=(x^v)^\sigma-x=(x^v)^\sigma-x^v+x^v-x=[x^v,\sigma]+[x,v].
\]

If $\<v,w\>=0$, then
\[
[x,v](w)=x^v(w)-x(w)=f^{\<v,w\>}(x(w))-x(w)=f^0(x(w))-x(w)=0.
\]

Finally, for each $\A_p$-orbit $S$ of $V_n-\{0\}$, 
\[
\sum_{w\in S}[x,\sigma](w)=(\sum_{w\in S}x^\sigma(w))-(\sum_{w\in S}x(w))=(\sum_{w\in S}x(w^{(\sigma^{-1})}))-(\sum_{w\in S}x(w))=(\sum_{w\in S}x(w))-(\sum_{w\in S}x(w))=0.
\] 
\end{proof}

\begin{prop}
$M_n\rtimes G_n$ has a commutator width of at least $\frac{n(p-1)}{(p!)}$.
\end{prop}
\begin{proof}
We know that $M_n=[M_n,G_n]$. We first investigate the diameter of $M_n$ with respect to the generators $[g,x]$ for $g\in G_n,x\in M_n$.

Fix a positive integer $k$. Suppose $x\in M_n$ has $x=[x_1,v_1 \sigma_1]+\dots+[x_k,v_k  \sigma_k]$ for some $\sigma_1,\dots,\sigma_k\in\A_p$, $v_1,\dots,v_k\in V_n$ and $x_1,\dots,x_k\in M_n$. 

Let $W_n$ be an $\A_p$-invariant subspace of $V_n$ orthogonal to all $v_i$. Thus by Proposition~\ref{prop:CodimWwithSmallerInvCodim}, it has codimension at most $k\frac{p!}{2}$. If $k<\frac{2n(p-1)}{p!}$, then $W_n$ is not trivial. Furthermore, since $W_n$ is $\A_p$-invariant, $W_n$ is a collection of $\A_p$-orbits.

By Lemma~\ref{lem:AVMCalculation}, we have
\[
x=[x_1,v_1 \sigma_1]+\dots+[x_k,v_k  \sigma_k]=\sum[x_i^{v_i},\sigma_i]+\sum[x_i,v_i].
\]

Now for each $i$, we have $W_n\subseteq\{v_i\}^\perp\subseteq \nul([x_i,v_i])$. Hence $W_n\subseteq\nul(\sum[x_i,v_i])$. 

Furthermore, for each $\A_p$-orbit $S$ of $V_n-\{0\}$, we have $\sum_{w\in S}[x_i^{v_i},\sigma_i](w)=0$. 

Therefore, for each $\A_p$-orbit $S$ inside of $W_n$, we have
\begin{align*}
\sum_{w\in S} \sum_i [x_i^{v_i},\sigma_i](w)&=0, \\
\sum_{w\in S} \sum_i[x_i,v_i](w)&=0. 
\end{align*}
So together, we see that $\sum_{w\in S} x(w)=0$.

In particular, let us pick any $x\in M_n$ such that $\sum_{w\in S} x(w)\neq 0$ for each $\A_p$-orbit $S$ of $V_n-\{0\}$. Then $x= [x_1,\sigma_1  v_1]+\dots+[x_k,\sigma_k  v_k]$ implies that the corresponding $W_n$ must be trivial. In particular, we must have $k\geq\frac{2n(p-1)}{p!}$.

So the collection of $[g,x]$ for $g\in G_n,x\in M_n$ generates $M_n$ with a diameter of at least $\frac{2n(p-1)}{p!}$. Thus the commutator width of our group is at least $\frac{n(p-1)}{(p!)}$ by Lemma~\ref{lem:[M,G]Diam} below.
\end{proof}

\begin{lem}
\label{lem:[M,G]Diam}
Let $M$ be an abelian group, and suppose we have a group $G$ acting on $M$. If the group $M\rtimes G$ has commutator width at most $d$, then the elements $[x,g]$ for all $x\in M,g\in G$ generate $M$ with a diameter of at most $2d$.
\end{lem}
\begin{proof}
We write the group operation on $M$ multiplicatively for the sake of this argument. Suppose $M\rtimes G$ has commutator width at most $d$.

Then pick any $x  g\in M\rtimes G$, it must be the product of $d$ commutators, say 
\[
x  g=\prod_{i=1}^{d} [x_i  g_i,x'_i  g'_i].
\] 
Then by taking all the $g_i,g'_i$ to the right, for some $r_i,s_i,p_i,q_i$ in $G$, we have 
\[
x  g=\prod_{i=1}^{d}((x_i^{-1})^{r_i}((x'_i)^{-1})^{s_i}x_i^{p_i}(x'_i)^{q_i}) (\prod_{i=1}^{d} [g_i,g'_i]).
\] 
For the two sides to be equal, we must have 
\[
x=\prod_{i=1}^{d}((x_i^{-1})^{r_i}((x'_i)^{-1})^{s_i}x_i^{p_i}(x'_i)^{q_i}).
\] 
By replacing $x_i,x'_i$ by $x_i^{r_i},(x'_i)^{s_i}$, and replacing $p_i,q_i$ by $r_i^{-1}p_i,s_i^{-1}q_i$, we have 
\[
x=\prod_{i=1}^{d}(x_i^{-1}(x'_i)^{-1}x_i^{p_i}(x'_i)^{q_i}).
\]

Note that $M$ is abelian. Then 
\[
x=\prod_{i=1}^{d}(x_i^{-1}(x'_i)^{-1}x_i^{p_i}(x'_i)^{q_i})=\prod_{i=1}^{d}([x_i,p_i][x'_i,q_i]).
\]

In particular, since $x$ is artibrary, $M$ must be generated by elements $[x,g]$ for all $x\in M,g\in G$ with a diameter of at most $2d$.
\end{proof}

In particular, as $n$ increases, we see that the commutator width of $M_n\rtimes G_n$ is unbounded as desired.

\subsection{Ultraproducts of $M_n$ and $M_n^*$}
\label{subsec:dual}

Let $V=\prod_{n\to\omega}V_n$, $M=\prod_{n\to\omega}M_n$, and let $M^*=\prod_{n\to\omega}M_n^*$. We shall establish some properties of them to be used later.

Note that each $M_n$ is the collection of set functions from $V_n-\{0\}$ to $B$. Hence given a sequence $(x_n)_{n\in\N}$, their ultralimit $\lim_{n\to\omega}x_n\in M$ is a set function from $\prod_{n\to\omega}(V_n-\{0\})=V-\{0\}$ to $B$. So $M$ is canonically a submodule of $B^{V-\{0\}}$.

\begin{rem}
We do NOT have $M=B^{V-\{0\}}$. Note that all $M_n$ are finite, so the cardinality of their ultraproduct $M$ by a non-principal ultrafilter is that of the continuum. Similarly, note that all $V_n$ are also finite, so the cardinality of $V$ is also that of the continuum. Then $B^{V-\{0\}}$ will have cardinality strictly larger than that of the continuum. In particular, $M\neq B^{V-\{0\}}$.
\end{rem}

This means we can treat each $x\in M$ as a set-function from $V-\{0\}$ to $B$. So we can make the following definition.

\begin{defn}
For each $x\in M$, we define its null set as 
\[
\nul(x):=\{0\}\cup\{v\in V-\{0\}\mid x(v)=0\}.\]
\end{defn}

\begin{lem}
\label{prop:nulxnulxn}
If $x=\lim_{n\to\omega}x_n$, then 
\[
\nul(x)=\prod_{n\to\omega}\nul(x_n)\subseteq V.\]
\end{lem}
\begin{proof}
Note that for each $v=\lim_{n\to\omega}v_n\in V-\{0\}$, then $x(v)=0$ if and only if 
\[\{n\in\N:x_n(v_n)=0\}\in\omega.\] 

But note that $x_n(v_n)=0$ implies that $v_n\in\nul(x_n)$, so we have 
\[\{n\in\N:v_n\in\nul(x_n)\}\in\omega.\] 

So we have
\[v\in\prod_{n\to\omega}\nul(x_n).\] 

Since this is true for all $v\in V-\{0\}$, we have 
\[\nul(x)\subseteq\prod_{n\to\omega}\nul(x_n).\]

The other direction is obvious.
\end{proof}

\begin{prop}
For any $x\in M$ and $v\in V$, then 
\[
\{v\}^\perp\subseteq\nul([x,v]).\]
\end{prop}
\begin{proof}
Say $v=\lim_{n\to\omega}v_n$ and $x=\lim_{n\to\omega}x_n$. Then 
\[
[x,v]=x^v-x=\lim_{n\to\omega}(x_n^{v_n}-x_n)=\lim_{n\to\omega}[x_n,v_n].
\] 
Since $\{v_n\}^\perp\subseteq\nul([x_n,v_n])$ by Lemma~\ref{lem:AVMCalculation}, we see that 
\[
\{v\}^\perp=\prod_{n\to\omega}\{v_n\}^\perp\subseteq\prod_{n\to\omega}\nul([x_n,v_n])=\nul([x,v]).
\]
\end{proof}

\begin{prop}
For any $x\in M$ and $v\in V$ and $w\in V-\{0\}$, then 
\[
x^v(w)=f^{\<v,w\>}(x(w)).\]
\end{prop}
\begin{proof}
Say $v=\lim_{n\to\omega}v_n$, $w=\lim_{n\to\omega}w_n$ and $x=\lim_{n\to\omega}x_n$. Then 
\[
x^v(w)=\lim_{n\to\omega}x_n^{v_n}(w_n)=\lim_{n\to\omega}f^{\<v_n,w_n\>}(x_n(w_n)).
\] 
But since we have
\begin{align*}
\{n\in\N:\<v_n,w_n\>=\<v,w\>\}&\in\omega,\\
\{n\in\N:x_n(w_n)=x(w)\}&\in\omega. 
\end{align*}
So their intersection is still in $\omega$, so we have 
\[
\lim_{n\to\omega}f^{\<v_n,w_n\>}(x_n(w_n))=f^{\<v,w\>}(x(w)).
\]
\end{proof}

Now we define the ``support for $\phi$'' as the dual concept of ``null set for $x$''.

\begin{defn}
For each $\phi_n\in M_n^*$, we say a subset $S$ of $V_n$ is a supporting set for $\phi_n$ if for all $x_n\in M_n$ such that $\nul(x_n)\supseteq S$, we have $\phi_n(x_n)=0$. Define $\supp(\phi_n)$ to be the intersection of all supporting sets for $\phi_n$.
\end{defn}

\begin{prop}
$\supp(\phi_n)$ is also a supporting set for all $\phi_n\in M_n^*$.
\end{prop}
\begin{proof}
First, the whole space $V_n$ is a supporting set for $\phi_n$.

For any two supporting sets for $\phi_n$, say $S_1,S_2$, suppose $S_1\cap S_2\subseteq\nul(x)$ for some $x\in M_n$. 

Let $x_1\in M_n$ be the function that agrees with $x$ on $S_1-S_2$, but sends all other inputs to zero. Then $S_2\subseteq\nul(x_1)$ and hence $\phi_n(x_1)=0$. 

Now we have 
\begin{align*}
S_1\cap S_2\subseteq&\nul(x),\\
S_1\cap S_2\subseteq&\nul(x_1).
\end{align*}
Therefore we have 
\[
S_1\cap S_2\subseteq\nul(x-x_1).
\] 
But we also see that 
\[
S_1-S_2\subseteq\nul(x-x_1).
\] 
So we can conclude that $S_1\subseteq\nul(x-x_1)$ and hence $\phi_n(x-x_1)=0$. 

Combining results above, we have 
\[
\phi_n(x)=\phi_n(x-x_1)+\phi_n(x_1)=0.
\] 
So finite intersections of supporting sets are supporting sets.

Finally, since $V_n$ is finite, therefore $\supp(\phi_n)$ is a finite intersection of supporting sets for $\phi_n$, so it is a supporting set for $\phi_n$.
\end{proof}

\begin{prop}
\label{prop:SuppPhiCriteria}
For any $v_n\in V_n$, then $v_n\in\supp(\phi_n)$ if and only if there exists $x_n\in M_n$ such that the followings are true
\begin{align*}
\nul(x_n)&=V_n-\{v_n\},\\
\phi_n(x_n)&\neq 0.
\end{align*}
\end{prop}
\begin{proof}
Suppose $v_n\in\supp(\phi_n)$. Then $V_n-\{v_n\}$ is NOT a supporting set for $\phi_n$. So there exists $x_n\in M_n$ such that we have the following 
\begin{align*}
\nul(x_n)&\supseteq V_n-\{v_n\},\\
\phi_n(x_n)&\neq 0.
\end{align*}
In particular, $\phi_n(x_n)\neq 0$ implies that $x_n\neq 0$. Therefore $\nul(x_n)\neq V_n$, and we must in fact have $\nul(x_n)=V_n-\{v_n\}$ as desired.

Conversely, suppose $v_n\notin\supp(\phi_n)$. Then for all $x_n\in M_n$ with $\nul(x_n)=V_n-\{v_n\}$, we have
\[
\supp(\phi_n)\subseteq\nul(x_n).\]
And thus 
\[
\phi_n(x_n)=0.\]
\end{proof}

Now consider $\phi=\lim_{n\to\omega}\phi_n\in M^*$. 
%Note that for each $x=\lim_{n\to\omega}x_n\in M$, we would have $\phi(x)=\lim_{n\to\omega}\phi_n(x_n)\in\Z/m\Z$. 
We make the following definition.

\begin{defn}
We say a subset $S$ of $V$ is a supporting set for $\phi$ if $\phi(x)=0$ whenever $\nul(x)\supseteq S$. Define $\supp(\phi)$ to be the intersection of all supporting sets for $\phi$.
\end{defn}

\begin{prop}
If $\phi=\lim_{n\to\omega}\phi_n$, then $\supp(\phi)$ is a supporting set for $\phi$ and is equal to $\prod_{n\to\omega}\supp(\phi_n)$.
\end{prop}
\begin{proof}
For $x=\lim_{n\to\omega}x_n\in M$, suppose \[\prod_{n\to\omega}\supp(\phi_n)\subseteq\nul(x)\].

Then by Proposition~\ref{prop:nulxnulxn}, we have

\[\prod_{n\to\omega}\supp(\phi_n)\subseteq\prod_{n\to\omega}\nul(x_n).\]

Then for $\omega$-almost all $n$, we have 
\[\supp(\phi_n)\subseteq\nul(x_n).\]

Thus $\phi_n(x_n)=0$ for $\omega$-almost all $n$. So $\phi(x)=0$. 

So the set $\prod_{n\to\omega}\supp(\phi_n)$ is a supporting set for $\phi$.

Conversely, for each $v_n\in\supp(\phi_n)$, then there exists $x_n\in M_n$ such that 
\begin{align*}
\nul(x_n)&\supseteq V_n-\{v_n\},\\
\phi_n(x_n)&\neq 0.
\end{align*}
Then for $v=\lim_{n\to\omega}v_n\in\prod_{n\to\omega}\supp(\phi_n)$ and $x=\lim_{n\to\omega}x_n$, we have
\begin{align*}
\nul(x)&\supseteq V-\{v\},\\
\phi(x)&\neq 0.
\end{align*}

So NO subsets of $V-\{v\}$ can be a supporting set for $\phi$. Hence any supporting set for $\phi$ must contain $v$. 

Since we started with arbitrary $v_n\in\supp(\phi_n)$, the statement above is true for arbitrary $v\in\prod_{n\to\omega}\supp(\phi_n)$. As a result, $\prod_{n\to\omega}\supp(\phi_n)$ is the unique smallest supporting set for $\phi$.
\end{proof}

\subsection{The existence of $\U$ and $\psi$}
\label{subsec:filter}

We now construct the homomorphism $\psi:\prod_{n\to\omega}M_n\to\Z/m\Z$ by picking the right ultrafilter $\U$ on $M^*$, and set $\psi$ to be the $\U$-consensus homomorphism $\phi_\U$. To facilitate computations, we shall first pick out bases for all the free $\Z/m\Z$-modules involved.

For $B$, set $b_1:=(1,-1,0,\dots, 0)$, and set $b_{i+1}:=f(b_i)$ for each $i\in\Z/q\Z$. Then we know that $b_1,\dots,b_{q-1}$ form a basis. We fix this basis for $B$ for this section.

For each $v\in V_n-\{0\}$ and each integer $1\leq i\leq q-1$, let $x_{v\to i}\in M_n$ be the function that sends $v$ to $b_i$ and all other inputs to zero.

\begin{prop}
The elements $x_{v\to i}$ for each $v\in V_n$ and each integer $1\leq i\leq q-1$ form a basis for $M_n$.
\end{prop}
\begin{proof}
Straightforward verification.
\end{proof}

We shall use the basis  $x_{v\to i}$ as a standard basis for $M_n$. For each $v\in V_n-\{0\}$, let $\phi_v:M_n\to\Z/m\Z$ such that $x_{v\to 1}$ is sent to $1$, and all other standard basis elements $x_{w\to i}$ for $(w,i)\neq (v,1)$ are sent to zero.

\begin{lem}
\label{lem:PhiVXSigma}
$\phi_v(x^\sigma)=\phi_{v^\sigma}(x)$ for all $x\in M_n$, $v\in V_n-\{0\}$, and $\sigma\in\A_p$.
\end{lem}
\begin{proof}
Note that $(x_{v\to i})^{(\sigma^{-1})}=x_{v^{\sigma}\to i}$ by definition. So the $x_{v\to 1}$-component of $x^\sigma$ is the same as the $x_{v^\sigma\to 1}$-component of $x$. So $\phi_v(x^\sigma)=\phi_{v^\sigma}(x)$.
\end{proof}

\begin{prop}
For any $v\in V_n-\{0\}$, we have 
\begin{enumerate}
\item $\nul(x_{v\to i})=V_n-\{v\}$.
\item $\supp(\phi_v)=\{v\}$.
\item For any homomorphism $\phi:M_n\to\Z/m\Z$, then $v\in\supp(\phi)$ if and only if $\phi(x_{v\to i})\neq 0$ for some $i$.
\end{enumerate}
\end{prop}
\begin{proof}
We have $\nul(x_{v\to i})=V_n-\{v\}$ by definition. 

For the second statement, note that $\phi_v(x_{v\to 1})=1$, so $v\in\supp(\phi_v)$. And for any $x\in M_n$ with $\{v\}\subseteq \nul(x)$, then all $x_{v\to i}$-components of $x$ are zero. So $\phi_v(x)=0$. So $\{v\}$ is the smallest supporting set for $\phi_v$.

For the last statement, if $\phi(x_{v\to i})\neq 0$ for some $i$, note that the null set for $x_{v\to i}$ is $V_n-\{v\}$. Then all subsets of $V_n-\{v\}$ cannot be a supporting set for $\phi$. So we must have $v\in\supp(\phi)$. 

Conversely, if $v\in\supp(\phi)$, then by Proposition~\ref{prop:SuppPhiCriteria}, for some $x\in M_n$ with $\nul(x)=V_n-\{v\}$, we have $\phi(x)\neq 0$. But the identity $\nul(x)=V_n-\{v\}$ implies that $x$ must be a linear combination of $x_{v\to 1},\dots,x_{v\to q-1}$, and thus $\phi(x_{v\to i})\neq 0$ for some $x_{v\to i}$.
\end{proof}

Now we start to build the homomorphism $\psi$. We want to guarantee that $\psi$ is surjective, so it must some specific element of $M$ to the generator $1$ of $\Z/m\Z$.

We define $z_n$ to be the element in $M_n$ such that 
\[
z_n:=\sum_{v\in V_n-\{0\}}x_{v\to 1}.\]

We then take an ultralimit and set 
\[z:=\lim_{n\in\infty}z_n\in M=\prod_{n\to\omega}M_n.\] 

Our goal is to construct $\psi:=\phi_\U$ such that it sends $z$ to $1$, and it is $G$-invariant. Let us now reduce this to the task of finding a nice ultrafilter $\U$ on $M^*$.

For any closed subspace $W\subseteq V$, let $M_W^*$ be the collection of all $\phi\in M^*$ such that $\supp(\phi)\subseteq W$.

Let $E^*$ be the collection of all $\phi\in M^*$ such that $\phi(z)=1$, and $\phi(x)=\phi(x^\sigma)$ for all $x\in M$ and $\sigma\in\A_p$. 

\begin{prop}
\label{prop:UltraUIsEnough}
Suppose $\U$ is an ultrafilter on $M^*$ such that $M_W^*\in\U$ for all closed subspace $W$ of finite codimension in $V$, and $E^*\in\U$. Then $\U$ is $G$-invariant and $\phi_\U:M\to\Z/m\Z$ is surjective and $G$-invariant.
\end{prop}
\begin{proof}
Note that for all $\phi\in E^*$, $\phi(z)=1$. Since $E^*\in\U$, we conclude that $\phi_\U(z)=1$. So the image of $\phi_\U$ contains a generator of $\Z/m\Z$, and thus $\phi_\U$ is surjective.

For any $g\in G=V\rtimes\A_p$, say $g=v \sigma$. Let $S_g$ be the set of all $g$-invariant elements $\phi\in M^*$. Let us show that $S_g\in\U$.

Set $W=\{v\}^\perp$, and consider any $\phi\in E^*\cap M_W^*$. Then since $\phi$ is in $E^*$, it is $\sigma$-invariant. 

Since $\phi$ is in $M_W^*$, for any $x\in M$, we have 
\[
\supp(\phi)\subseteq W\subseteq \nul([x,v]).
\] 
Thus we have 
\[
0=\phi([x,v])=\phi(x^v-x).
\]
Therefore $\phi(x)=\phi(x^v)$. Hence $\phi$ is $v$-invariant. 

In conclusion, $\phi$ is both $\sigma$-invariant and $v$-invariant, therefore it is $g$-invariant. So $E^*\cap M_W^*\subseteq S_g$. But since $E^*\cap M_W^*\in\U$, therefore we see that $S_g\in\U$. 

So $\U$ is $G$-invariant, and consequently, $\phi_\U$ is also $G$-invariant.
\end{proof}

In particular, the proof of Lemma~\ref{lem:CyclicLem} depends entirely on the existence of an ultrafilter containing $E^*$ and $M_W^*\in\U$ for all closed subspace $W$ of finite codimension in $V$. So we need to show that $E^*$ and these $M_W^*$ have the finite intersection property.

\begin{lem}
Let $W_n$ be an $\A_p$-invariant subspace of $V_n$ of codimension $c$, such that $n>c(p-1)$. Then there exists $\phi\in M_n^*$ such that for all $x\in M_n$ and $\sigma\in\A_p$, we have
\begin{align*}
\phi(z_n)&=1, \\
\supp(\phi)&\subseteq W_n, \\
\phi(x)&=\phi(x^\sigma).
\end{align*}
\end{lem}
\begin{proof}
Since $n>c(p-1)$, by Lemma~\ref{lem:WnOrbitp}, we can find $w\in W_n$ whose $\A_p$-orbit $\O(w)$ has exactly $p$ elements. Let $\phi=\sum_{v\in\O(w)}\phi_v$. 

Since $W_n$ is $\A_p$-invariant, we have $\O(w)\subseteq W_n$, and thus 
\[
\supp(\phi)=\O(w)\subseteq W_n.\]

By Lemma~\ref{lem:PhiVXSigma}, for all $x\in M_n$, $v\in V_n-\{0\}$, and $\sigma\in\A_p$, we have 
\[
\phi_v(x^\sigma)=\phi_{v^\sigma}(x).
\] 
So we see that for all $x\in M_n$ and $\sigma\in\A_p$, we have
\[\phi(x)=\sum_{v\in\O(w)}\phi_{v}(x)=\sum_{v\in\O(w)}\phi_{v^\sigma}(x)=\sum_{v\in\O(w)}\phi_v(x^\sigma)=\phi(x^\sigma).\]

Finally, note that $\phi_v(z_n)=1$ for all $v\in V_n-\{0\}$. So we have
\[\phi(z_n)=|\O(w)|=p.\] 
But note that $p,m$ are coprime, so $p$ is invertible in $\Z/m\Z$. So by replacing $\phi$ with a suitable multiple, we can guarantee that $\phi(z_n)=1$.
\end{proof}

\begin{prop}
\label{prop:MWEIntersectNonempty}
For any closed subspace $W$ of finite codimension in $V$, then $M_W^*\cap E^*$ is non-empty.
\end{prop}
\begin{proof}
Note that if $W$ is closed in $V$ with finite codimension, then by Proposition~\ref{prop:CodimWwithSmallerInvCodim}, we can find $W'\subseteq W$ which is closed, $\A_p$-invariant, and has finite codimension in $V$. We also necessarily have $M_{W'}^*\subseteq M_W^*$. So, replacing $W$ by the smaller subspace $W'$, we can assume that $W$ is also $\A_p$-invariant.

By Proposition~\ref{prop:ClosedSubspaceAreUltra}, since $W$ is $\A_p$-invariant and closed with finite codimension in $V$, therefore $W=\prod_{n\to\omega}W_n$ where each $W_n$ is $\A_p$-invariant and closed with bounded codimension in $V_n$ for all $n$. Say the codimension are all at most $c$. Then for all $n>c(p-1)$, we can find $\phi_n\in M_n^*$ such that for all $x\in M_n$ and $\sigma\in\A_p$, we have
\begin{align*}
\phi_n(z_n)&=1, \\
\supp(\phi_n)&\subseteq W_n, \\
\phi_n(x)&=\phi_n(x^\sigma).
\end{align*}

Then since $\omega$ is not principal, we only need to consider the cases where $n>c(p-1)$ as above. So this defines $\phi=\lim_{n\to\omega}\phi_n\in M^*$ such that for all $x=\lim_{n\to\omega}x_n\in M$ and $\sigma\in\A_p$,
\begin{align*}
\phi(z)&=\lim_{n\to\omega}\phi_n(z_n)=1, \\
\supp(\phi)&=\prod_{n\to\omega}\supp(\phi_n)\subseteq \prod_{n\to\omega}W_n=W, \\
\phi(x)&=\lim_{n\to\omega}\phi_n(x_n)=\lim_{n\to\omega}\phi_n(x_n^\sigma)=\phi(x^\sigma).
\end{align*}

So $\phi\in M_W^*\cap E^*$. So we are done.
\end{proof}

\begin{prop}
\label{prop:UltraUExists}
There is an ultrafilter $\U$ on $M^*$ such that $M_W^*\in\U$ for all closed subspace $W$ of finite codimension in $V$, and $E^*\in\U$.
\end{prop}
\begin{proof}
We only need to show that these subsets of $M^*$ have finite intersection property. Pick $W_1,\dots,W_k$ closed subspaces of finite codimension in $V$, we want to show that $E^*\cap(\bigcap_{i=1}^k M_{W_i}^*)$ is non-empty.

Let $W=\bigcap_{i=1}^k W_i$. This is a finite intersection of closed subspaces of finite codimensions, so it is itself closed with finite codimension.

Note that 
\begin{align*}
\phi\in\bigcap_{i=1}^k M_{W_i}^*\text{ if and only if }&\supp(\phi)\subseteq W_i\text{ for all }i,\\
\text{ if and only if }&\supp(\phi)\subseteq \bigcap_{i=1}^k W_i=W,\\
\text{ if and only if }&\phi\in M_{W}^*.
\end{align*}

So we have $\bigcap_{i=1}^k M_{W_i}^*=M_{W}^*$.

So we only need to show that $E^*\cap M_{W}^*$ is non-empty. But since $W$ is closed with finite codimension, this is non-empty by Proposition~\ref{prop:MWEIntersectNonempty}. So we have the desired finite intersection property, and the desired ultrafilter exists.
\end{proof}

We now finish the proof of Lemma~\ref{lem:CyclicLem}.

\begin{proof}[Proof of Lemma~\ref{lem:CyclicLem}]
Let $\U$ be the ultrafilter in Proposition~\ref{prop:UltraUExists}. Then it satisfies all the requirements of Proposition~\ref{prop:UltraUIsEnough}, so $\phi_\U$ is surjective and $G$-invariant. 

Furthermore, by Proposition~\ref{prop:MnSemiGnPerfect}, the groups $M_n\rtimes G_n$ are finite perfect groups as desired.
\end{proof}

\bibliographystyle{alpha}
\bibliography{references}
\end{document}